\documentclass{article}
\usepackage{amsmath,amssymb,amsthm,epsfig}
\usepackage{graphicx,float,url,color}
\usepackage[colorlinks=true]{hyperref}
\usepackage{pdfsync}
\usepackage{pgfplots}
\pgfplotsset{compat=newest}
\usetikzlibrary{shapes}
\usetikzlibrary{plotmarks}
\usepackage{tikz}
\usepackage{caption}
\usepackage{dsfont}
\usepackage{algorithmic}
\usepackage{lipsum}
\usepackage{amsfonts}
\usepackage{graphicx}
\usepackage{algorithm}

\usepackage{subfigure}
\usepackage{comment}
\graphicspath{{fig/}}

\def\RR{\mathbb{R}}
\def\NN{\mathbb{N}}
\def\eps{\varepsilon}

\def\herd{\mathrm{herd}}
\def\cH{{\mathcal H}}
\def\cP{{\mathcal P}}
\def\cR{{\mathcal R}}
\def\cU{{\mathcal U}}
\definecolor{mygreen}{rgb}{0,0.7,0}
\definecolor{myblue}{rgb}{0,0,1}
\definecolor{myred}{rgb}{1,0,0}

\renewcommand{\geq}{\geqslant}
\renewcommand{\leq}{\leqslant}

\newtheorem{theorem}{Theorem}

\newtheorem{proposition}{Proposition}

\newtheorem{lemma}{Lemma}

\theoremstyle{definition}\newtheorem{remark}{Remark}

\title{\bf Optimal immunity control by social distancing for the SIR epidemic model}

\author{P.-A. Bliman\thanks{Inria, Sorbonne Universit\'e, Universit\'e Paris-Diderot SPC, CNRS,
Laboratoire Jacques-Louis Lions, \'equipe Mamba, Paris, France.
  (\texttt{pierre-alexandre.bliman@inria.fr}).}
\and M. Duprez\thanks{CEREMADE, Universit\'e Paris-Dauphine \& CNRS UMR 7534, Universit\'e PSL, 75016 Paris, France 
  (\texttt{mduprez@math.cnrs.fr}).}
\and Y. Privat\thanks{Universit\'e de Strasbourg, CNRS UMR 7501, INRIA, Institut de Recherche Math\'ematique Avanc\'ee (IRMA), 7 rue Ren\'e Descartes, 67084 Strasbourg, France
  (\texttt{yannick.privat@unistra.fr}).}
\and N. Vauchelet\thanks{LAGA, UMR 7539, CNRS, Universit\'e Sorbonne Paris Nord,
99 avenue Jean-Baptiste Cl\'ement, 93430 Villetaneuse, France
  (\texttt{vauchelet@math.univ-paris13.fr}).}
  }

\date{}

\begin{document}

\maketitle

\begin{abstract}
 Until a vaccine or therapy is found against the SARS-CoV-2 coronavirus, reaching herd immunity appears to be the only mid-term option.
However, if the number of infected individuals decreases and eventually fades only beyond this threshold, a significant proportion of susceptible  
may still be infected until the epidemic is over.
A containment strategy  is likely the best policy in the worst case where no vaccine or therapy is found.
In order to keep the number of newly infected persons to a minimum, a possible strategy
is to apply strict containment measures, so that the number of susceptible individuals remains close to herd immunity.
Such an action is unrealistic since containment can only last for a finite amount of time and is never total. In this article, using a classical SIR model,  
we determine the (partial or total) containment strategy on a given finite time interval that maximizes the  number of susceptible individuals over an infinite horizon, or equivalently that minimizes the total infection burden during the curse of the epidemic.
The existence and uniqueness of the optimal strategy is proved and the latter is fully characterized.
If applicable in practice, such a strategy would lead theoretically to an increase by 30\% of the proportion of susceptible 
on an infinite horizon, for a containment level corresponding to the sanitary measures put in place in France from  March to May 2020.
 We also analyze the minimum intervention time to reach a fixed distance from herd immunity, and show the relationship with the previous problem. Simulations are provided that  illustrate and validate the theoretical results.
\end{abstract}


\medskip



\section{Introduction}

Important efforts have been devoted to the application of optimal control theory to the control of infectious diseases.
Many works have considered the simple SIR model without demography that recently became even more famous in a broad community worldwide.
Such a model reads
\begin{subequations}
\label{eq1}
\begin{gather}
\label{eq1a}
\dot S = -\beta SI,\qquad S(0)=S_0,\\
\label{eq1b}
\dot I = \beta SI - \nu I,\qquad I(0) = I_0,\\
\label{eq1c}
\dot R = \nu I,\qquad R(0)=R_0.
\end{gather}
\end{subequations}
The state variables $S$, $I$, $R$ correspond respectively to the proportions of susceptible, infected and removed individuals in the population. 
Notice that the sum of the derivatives of the three state variables is zero, so the sum of the variables remains equal to 1 if initially $S_0+I_0+R_0=1$. 
One may therefore describe the system solely with the equations \eqref{eq1a}-\eqref{eq1b}.
The total population is constant and 
no demographic effect (births, deaths) is modeled, as they are not relevant to the time scale to be taken into account in reacting to an outbreak.
The positive parameter $\beta$ in the infection term $\beta SI$ accounts at the same time for the rate of encounters between the individuals and the probability of transmitting the infection during each of theses encounters.
The positive parameter $\nu$ describes the rate at which the infected are removed (typically by recovery) by time unit.
Cured people are supposed to have acquired permanent immunity.

An important issue, abundantly discussed these days, is {\em herd immunity}.
The latter occurs naturally when a large proportion of the population has become immune to the infection.
Mathematically, it is defined as the value of $S$ below which the number of infected decreases.
For the SIR model \eqref{eq1}, one has $\dot I  = (\beta S - \nu)I$, and this number reads
\begin{equation}
\label{eq33}
S_\herd := \frac{\nu}{\beta}.
\end{equation}
While the number of infected decreases when $S(\cdot) \leq S_\herd$, epidemics continue to consume susceptible and to generate new infections once the immunity threshold has been crossed.

The {\em basic reproduction number}
$\cR_0 := \beta/\nu = \frac{1}{S_\herd}$
governs the behaviour of the system departing from its initial value.
When $\mathcal{R}_0<1$, no epidemic may occur, whereas when $\mathcal{R}_0>1$ epidemic occurs if $S_0> \frac{1}{\cR_0}$.
In such a case, $I$ reaches a peak and then goes to zero due to the immunity effect previously presented.
From now on, we assume
$S_0>S_\herd$ and
\begin{equation*}
\mathcal{R}_0 = \frac{\beta}{\nu} > 1.
\end{equation*}

Obviously, every solution of \eqref{eq1} is nonnegative and thus, $\dot R \geq 0 \geq \dot S$ at any instant, so $S$ may only decrease, while $R$ may only increase.
We infer that $S$  remains bounded and therefore,  the integral $\int_0^{+\infty} \beta S(t)I(t)\ dt$ converges, and in any case the following limit exists:
\begin{equation}
\label{eq41}
S_\infty := \lim_{t\to +\infty} S(t) \in [0,S_\herd],\qquad \lim_{t\to +\infty} I(t) = 0.
\end{equation}

We remark that, after passing the collective immunity threshold, that is while $t$ is large enough, one has 
\begin{equation}\label{eq:bound S}
S_\herd \geq S(t)\geq S_\infty.
\end{equation}

In order to illustrate herd immunity, Table \ref{fi1} displays, for several values of $\cR_0$,  the value of the herd immunity threshold and the number of  susceptibles that remain after the fading out of the outbreak, in the case of an initially naive population ($S_0\simeq 1$).
The proportion of infections occurred after the overcome of the immunity is also shown. 
\begin{table}[ht]
\begin{center}
\begin{tabular}{|c|cccccc|}
\hline
$\cR_0$ & 1.5 & 2 & 2.5 & \textbf{2.9}&3 & 3.5\\
\hline
$S_\herd$ & 0.67 & 0.50 & 0.40 &\textbf{0.34}& 0.33 & 0.29\\
\hline
$S_\infty$
& 0.42 & 0.20 & 0.11 &\textbf{0.067}& 0.059 & 0.034\\
\hline
$\displaystyle\frac{S_\herd-S_\infty}{1-S_\infty}$
& 43\% & 37\% & 33\% &\textbf{30\%}& 29\% & 27\% \\
\hline
\end{tabular}

\caption{Herd immunity level $S_\herd$ and asymptotic susceptible proportion $S_\infty$ for initial value $S_0\simeq 1$ for several values of the basic reproduction number $\cR_0$.
The value of $S_\herd$ comes from \eqref{eq33}, and $S_\infty$ may be deduced from the fact that $S_\infty -\frac{1}{\cR_0}\ln S_\infty \approx 1$, see Lemma \ref{co1} below.
The ratio $\frac{S_\herd-S_\infty}{1-S_\infty}$ represents the proportion of susceptible that occur after passing the collective immunity threshold. The column in bold corresponds to $\mathcal{R}_0$ found in \cite{Salje:2020aa} for the SARS-CoV-2 in France before the lockdown of March-May 2020.
\label{fi1}}
\end{center}
\end{table}

Apart from medical treatment, there are generally speaking three main methods to control human diseases.
Each of them, alone or in conjunction with the others, gave rise to applications of optimal control.
We introduce a modified SIR system, where the control inputs $u_1(\cdot)$, $u_2(\cdot)$, $u_3(\cdot)$ materialize these three methods.
\begin{subequations}
\label{eq11}
\begin{gather}
\label{eq11a}
\dot S = -u_3\beta SI - u_1,\qquad S(0)=S_0,\\
\label{eq11b}
\dot I = u_3\beta SI - \nu I - u_2,\qquad I(0) = I_0,\\
\label{eq11c}
\dot R = \nu I + u_1 + u_2,\qquad R(0)=R_0.
\end{gather}
\end{subequations}

A first class of interventions consists in vaccination or immunization \cite{Abakuks:1974aa,Morton:1974aa,Di-Blasio:1980aa,Greenhalgh:1988aa,Behncke:2000aa,Gaff:2009aa,Hansen:2011aa,Hu:2013aa,Alkama:2014aa,Yang:2015aa,Laguzet:2015aa,Bolzoni:2017aa,Bolzoni:2019aa,Shim:2019aa}.
It consists in transferring individuals from the $S$ compartment to the $R$ one: this is the effect of the additive input $u_1$ in \eqref{eq11}.
The members of the latter may not be anymore ``recovered" stricto sensu, but the key point in the model is that they are excluded from the infective process.
A second class of measures corresponds to screening and quarantining of infected \cite{Abakuks:1973aa,Wickwire:1975aa,Wickwire:1979aa,Behncke:2000aa,Ainseba:2012aa,Zhou:2013aa,Bolzoni:2017aa,Bolzoni:2019aa}.
It may be modelled by transfer of individuals from the $I$ compartment to the $R$ one, as done by the input $u_2$ in \eqref{eq11}.
Last, it is possible to reduce transmission through health promotion campaigns or lockdown policies \cite{Behncke:2000aa,Bolzoni:2017aa,Morris:2020aa}, as done by the multiplicative input $u_3$ in \eqref{eq11}.
Of course these methods may be employed jointly \cite{Hansen:2011aa}.

Other modelling frameworks have also been considered. More involved models called
SEIR and SIRS have been analyzed in \cite{Behncke:2000aa,Gaff:2009aa}, and SIR model structured by the infection age in \cite{Ainseba:2012aa}.
Constraints on the number of infected persons that the public health system can accommodate or on available resources, particularly in terms of vaccination, were studied \cite{Zhou:2013aa,Yang:2015aa,Bolzoni:2017aa,Bolzoni:2019aa}.
Economical considerations may be aggregated to the epidemiological model \cite{Kruse:2020aa,Piguillem:2020aa,Alvarez:2020aa}.
Ad-hoc models for tackling emergence of resistance to drugs issues have been introduced \cite{Jaberi-Douraki:2013aa,Jaberi-Douraki:2013ab}, as well as framework allowing to study revaccination policies \cite{Kolesin:2015aa,Kolesin:2016aa}.
Optimization of vaccination campaigns for vector-borne diseases have also been considered \cite{Shim:2019aa}.
The references cited above are limited to deterministic differential models, but discrete-time models and stochastic models have also been used e.g. in \cite{Abakuks:1973aa,Abakuks:1974aa,Greenhalgh:1988aa}.

The costs considered in the literature are usually integral costs combining an ``outbreak size" (the integral of the number of infected, or of the newly infected term, or the largest number of infected) and the input variable on a given finite time horizon.
Some minimal time control problems have also been considered \cite{Alkama:2014aa,Bolzoni:2017aa,Bolzoni:2019aa}.
Few results consider infinite horizon \cite{Behncke:2000aa}. 
Qualitatively, optimal solutions attached to the vaccination or isolation protocols are in general bang-bang\footnote{In other words, they only take a.e.\ two different values.}, with an intervention from the very beginning.
Bounds on the number of switching times between the two modes (typically zero or one) are sometimes provided.
By contrast, protocols of the third type (health promotion campaigns or lockdown policies) usually provide bang-bang optimal solutions with transmission reduction beginning after a certain time.

The present article is dedicated to the optimal control issue of obtaining, by enforcing social distancing, the largest value for $S_\infty$, the limit number of susceptible individuals at infinity. This means that we act on System~\eqref{eq11} through the function $u_3(\cdot)$ and set $u_1(\cdot)=u_2(\cdot)=0$.
This aim is also equivalent to minimize the total infection burden $\int_0^{+\infty} \beta  u_3(t)S(t)I(t)\ dt$ or, equivalently, to minimize the final number of removed $\nu\int_0^TI(t)dt$, see \eqref{eq11a}-\eqref{eq11b}.
We thus seek here to determine how close to the herd immunity threshold it is possible to stop the spread of the disease, in the case where no vaccine or treatment is found to modify its evolution.

In the case of interventions on a infinite horizon, a possible action is to let the susceptible proportion reach the collective immunity level ($u_3=1$), and then impose total lockdown ($u_3=0$).
This situation is illustrated in Figure \ref{fig1} (right) for parameters of the Covid-19 in France in 2020 borrowed from \cite{Salje:2020aa} and given in Table \ref{tab:value} below.
We will focus here on interventions on a given finite horizon through possibly partial lockdown, i.e.\ with $u_3$ bounded from below.
This setting is a prototypal problem, voluntarily unrealistic since for instance, no constraint on the peak value of the number of infected is considered.
In a nutshell, the aim is to determine what  best result may be obtained in terms of reduction of the total cumulative number of infected individuals, by applying lockdown of given maximal intensity and duration, in the worst conditions where no medical solution is discovered to stop earlier the epidemic spread. This models public interventions on the transmission rate by measures like social distanciation, restraining order, lockdown and so on, imposed on finite time horizon.

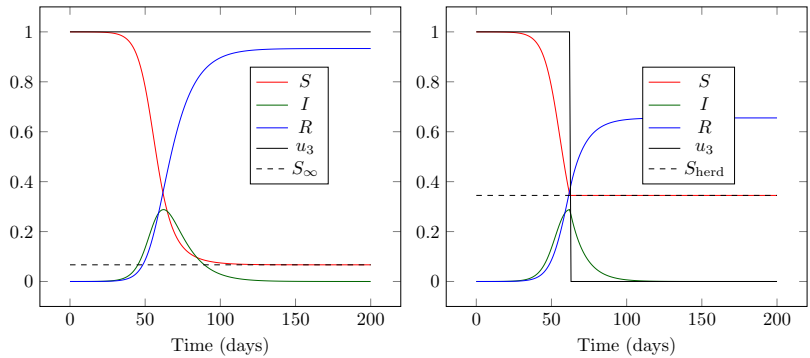
\begin{figure}[h] 
\begin{center}
\begin{tikzpicture}[thick,scale=0.7, every node/.style={scale=1.0}] \begin{axis}[xlabel=Time (days),,legend style={at={(0.8,0.8)}}]
\addplot[color=red] coordinates { 
(0.0,0.9999850746268657)
(1.0,0.9999803078410416)
(2.0,0.9999745436707488)
(3.0,0.9999675734463866)
(4.0,0.9999591448501336)
(5.0,0.9999489527900263)
(6.0,0.9999366283679209)
(7.0,0.9999217255440862)
(8.0,0.9999037050188164)
(9.0,0.999881914752167)
(10.0,0.9998555664233847)
(11.0,0.9998237069877883)
(12.0,0.9997851843160165)
(13.0,0.9997386056931209)
(14.0,0.9996822877064497)
(15.0,0.999614195754037)
(16.0,0.9995318710506894)
(17.0,0.999432342587359)
(18.0,0.999312020999904)
(19.0,0.9991665707144549)
(20.0,0.9989907560463698)
(21.0,0.998778256127011)
(22.0,0.998521442608034)
(23.0,0.9982111130418546)
(24.0,0.9978361716627374)
(25.0,0.9973832480125634)
(26.0,0.9968362425075228)
(27.0,0.9961757866992589)
(28.0,0.9953786047695257)
(29.0,0.9944167619087558)
(30.0,0.9932567849702414)
(31.0,0.9918586416166849)
(32.0,0.9901745667425094)
(33.0,0.9881477301913033)
(34.0,0.9857107489634677)
(35.0,0.982784061904316)
(36.0,0.9792742074034005)
(37.0,0.9750720774512845)
(38.0,0.9700512672247988)
(39.0,0.9640667006660917)
(40.0,0.9569537905431716)
(41.0,0.94852848471399)
(42.0,0.938588652087455)
(43.0,0.9269173570719071)
(44.0,0.9132886332698913)
(45.0,0.8974763545798369)
(46.0,0.8792666600235295)
(47.0,0.8584740576134586)
(48.0,0.8349607667482306)
(49.0,0.8086580589887477)
(50.0,0.7795874111061434)
(51.0,0.7478783958711802)
(52.0,0.7137797167218277)
(53.0,0.6776599858855513)
(54.0,0.6399959825037674)
(55.0,0.6013481592647179)
(56.0,0.5623256743894107)
(57.0,0.5235455012337112)
(58.0,0.4855914712077633)
(59.0,0.44897899599784874)
(60.0,0.41412975694395066)
(61.0,0.38135837401688677)
(62.0,0.35087071297606187)
(63.0,0.3227716860991776)
(64.0,0.29707946337174546)
(65.0,0.27374291363364966)
(66.0,0.2526595932982766)
(67.0,0.23369238197107217)
(68.0,0.21668366856844518)
(69.0,0.2014666588318789)
(70.0,0.18787384387678288)
(71.0,0.17574294582900812)
(72.0,0.16492078092424553)
(73.0,0.1552655015780045)
(74.0,0.14664764013317919)
(75.0,0.13895031038636227)
(76.0,0.13206884941093353)
(77.0,0.125910113373877)
(78.0,0.12039158239543259)
(79.0,0.11544038252533476)
(80.0,0.11099229697755249)
(81.0,0.10699081228991603)
(82.0,0.10338622620165924)
(83.0,0.1001348310122308)
(84.0,0.09719817749381503)
(85.0,0.09454241886841597)
(86.0,0.09213773099736873)
(87.0,0.08995780307342979)
(88.0,0.08797939225002568)
(89.0,0.08618193543183582)
(90.0,0.08454721163764609)
(91.0,0.08305904876161797)
(92.0,0.08170306909064272)
(93.0,0.08046646851080222)
(94.0,0.07933782490997768)
(95.0,0.07830693182966172)
(96.0,0.07736465392266353)
(97.0,0.07650280122822545)
(98.0,0.07571401968069134)
(99.0,0.07499169562387947)
(100.0,0.0743298724140111)
(101.0,0.07372317746354537)
(102.0,0.07316675831096844)
(103.0,0.0726562265018263)
(104.0,0.07218760823817322)
(105.0,0.07175730090089198)
(106.0,0.07136203467539537)
(107.0,0.07099883861902374)
(108.0,0.07066501060062697)
(109.0,0.07035809062161562)
(110.0,0.07007583709515351)
(111.0,0.06981620571781183)
(112.0,0.06957733061735792)
(113.0,0.06935750750263357)
(114.0,0.0691551785777377)
(115.0,0.06896891901385531)
(116.0,0.06879742479882363)
(117.0,0.0686395018075484)
(118.0,0.06849405595621405)
(119.0,0.06836008432034729)
(120.0,0.06823666711157693)
(121.0,0.06812296042073193)
(122.0,0.06801818964600746)
(123.0,0.06792164353455729)
(124.0,0.06783266877424228)
(125.0,0.06775066507955509)
(126.0,0.0676750807221045)
(127.0,0.06760540846159997)
(128.0,0.06754118183814738)
(129.0,0.06748197179093442)
(130.0,0.06742738357213762)
(131.0,0.06737705392818154)
(132.0,0.06733064852339668)
(133.0,0.06728785958369125)
(134.0,0.06724840374013084)
(135.0,0.06721202005433626)
(136.0,0.06717846820940374)
(137.0,0.06714752685164471)
(138.0,0.06711899206986414)
(139.0,0.06709267600016511)
(140.0,0.06706840554540071)
(141.0,0.06704602119941006)
(142.0,0.06702537596708727)
(143.0,0.0670063343721494)
(144.0,0.0669887715452061)
(145.0,0.06697257238539747)
(146.0,0.06695763078946422)
(147.0,0.06694384894265495)
(148.0,0.06693113666636533)
(149.0,0.06691941081784471)
(150.0,0.06690859473770884)
(151.0,0.06689861774136095)
(152.0,0.06688941465075392)
(153.0,0.06688092536322852)
(154.0,0.06687309445443515)
(155.0,0.06686587081259751)
(156.0,0.06685920730160325)
(157.0,0.06685306045061465)
(158.0,0.06684739016808294)
(159.0,0.0668421594782214)
(160.0,0.06683733427815279)
(161.0,0.0668328831140904)
(162.0,0.06682877697504588)
(163.0,0.06682498910267741)
(164.0,0.06682149481600488)
(165.0,0.06681827134981944)
(166.0,0.06681529770570974)
(167.0,0.06681255451471253)
(168.0,0.06681002391067473)
(169.0,0.06680768941348686)
(170.0,0.06680553582141324)
(171.0,0.0668035491118077)
(172.0,0.06680171634955764)
(173.0,0.06680002560265286)
(174.0,0.06679846586432188)
(175.0,0.06679702698122388)
(176.0,0.06679569958722215)
(177.0,0.06679447504230573)
(178.0,0.0667933453762561)
(179.0,0.0667923032366905)
(180.0,0.0667913418411402)
(181.0,0.06679045493285028)
(182.0,0.06678963674001065)
(183.0,0.06678888193815181)
(184.0,0.06678818561545956)
(185.0,0.06678754324078047)
(186.0,0.06678695063411079)
(187.0,0.06678640393937402)
(188.0,0.06678589959931089)
(189.0,0.06678543433231628)
(190.0,0.06678500511107334)
(191.0,0.06678460914284408)
(192.0,0.0667842438512883)
(193.0,0.06678390685969249)
(194.0,0.06678359597549867)
(195.0,0.06678330917603323)
(196.0,0.06678304459534148)
(197.0,0.06678280051204327)
(198.0,0.06678257533812974)
(199.0,0.06678236760862832)
(200.0,0.06678217597206923)

 };
\addplot[color=black!60!green] coordinates { 
(0.0,1.4925373134328359e-05)
(1.0,1.804841125735518e-05)
(2.0,2.182489175077052e-05)
(3.0,2.63915209071741e-05)
(4.0,3.191359816198464e-05)
(5.0,3.859099328096777e-05)
(6.0,4.666537151792117e-05)
(7.0,5.642892660323331e-05)
(8.0,6.823493520935744e-05)
(9.0,8.251051130643704e-05)
(10.0,9.977201674258135e-05)
(11.0,0.00012064367802745847)
(12.0,0.00014588007168688303)
(13.0,0.00017639327522106743)
(14.0,0.00021328564173078715)
(15.0,0.0002578893483591986)
(16.0,0.00031181409711803583)
(17.0,0.0003770046172646669)
(18.0,0.0004558099374304345)
(19.0,0.0005510667695784868)
(20.0,0.0006661997816595564)
(21.0,0.0008053420364521285)
(22.0,0.0009734794429774349)
(23.0,0.001176623702047641)
(24.0,0.0014220189193665533)
(25.0,0.0017183877864745906)
(26.0,0.0020762239512571883)
(27.0,0.002508137846919364)
(28.0,0.0030292637106401894)
(29.0,0.003657735629475884)
(30.0,0.004415239944857443)
(31.0,0.005327649853355174)
(32.0,0.00642574502633137)
(33.0,0.007746013794110007)
(34.0,0.00933152690477447)
(35.0,0.011232858781659737)
(36.0,0.013509012973652259)
(37.0,0.016228281293172213)
(38.0,0.019468929124241516)
(39.0,0.02331955115516111)
(40.0,0.027878882231216005)
(41.0,0.03325477968389297)
(42.0,0.03956202361442529)
(43.0,0.046918524702852446)
(44.0,0.05543950960713969)
(45.0,0.06522930771178088)
(46.0,0.07637053503533946)
(47.0,0.08891080725249678)
(48.0,0.10284764163765556)
(49.0,0.11811290756203223)
(50.0,0.1345589565203705)
(51.0,0.15194920374804796)
(52.0,0.16995615956865226)
(53.0,0.18816943235846784)
(54.0,0.20611489525433138)
(55.0,0.2232841556488694)
(56.0,0.23917116021459137)
(57.0,0.25331090924481636)
(58.0,0.26531450213109914)
(59.0,0.27489540180753075)
(60.0,0.281883689190879)
(61.0,0.28622755486113455)
(62.0,0.2879835846153755)
(63.0,0.28729895721243753)
(64.0,0.2843892617866061)
(65.0,0.2795153763760793)
(66.0,0.27296204402000646)
(67.0,0.265019787009967)
(68.0,0.25597087900436605)
(69.0,0.24607939266606166)
(70.0,0.2355848943554183)
(71.0,0.22469913653547682)
(72.0,0.2136050430412176)
(73.0,0.20245732996164312)
(74.0,0.19138420468289502)
(75.0,0.1804897021954491)
(76.0,0.1698563297463791)
(77.0,0.15954778764382874)
(78.0,0.149611611836895)
(79.0,0.14008164319585592)
(80.0,0.13098027157691472)
(81.0,0.12232043282321932)
(82.0,0.11410735684044565)
(83.0,0.10634007742185056)
(84.0,0.0990127217231795)
(85.0,0.09211560085618939)
(86.0,0.0856361242124664)
(87.0,0.07955955973558894)
(88.0,0.0738696610532009)
(89.0,0.06854918058679214)
(90.0,0.06358028575757782)
(91.0,0.05894489338141716)
(92.0,0.05462493540303016)
(93.0,0.05060256732070263)
(94.0,0.04686032902634853)
(95.0,0.04338126634104585)
(96.0,0.04014902025930632)
(97.0,0.03714788981550795)
(98.0,0.03436287353853167)
(99.0,0.03177969364964287)
(100.0,0.02938480646793034)
(101.0,0.027165401901789243)
(102.0,0.02510939440984299)
(103.0,0.02320540739757542)
(104.0,0.02144275266546975)
(105.0,0.019811406230710163)
(106.0,0.018301981598891546)
(107.0,0.0169057013572877)
(108.0,0.015614367790691938)
(109.0,0.014420333079243642)
(110.0,0.0133164695203773)
(111.0,0.012296140120166383)
(112.0,0.011353169819572972)
(113.0,0.010481817555662009)
(114.0,0.009676749304339592)
(115.0,0.008933012207644322)
(116.0,0.008246009853391299)
(117.0,0.007611478746637947)
(118.0,0.007025465989833981)
(119.0,0.0064843081706479995)
(120.0,0.005984611442504869)
(121.0,0.0055232327721298794)
(122.0,0.0050972623203004615)
(123.0,0.004704006916071776)
(124.0,0.004340974580566254)
(125.0,0.0040058600536653985)
(126.0,0.003696531275335396)
(127.0,0.0034110167726267558)
(128.0,0.003147493903419867)
(129.0,0.002904277908586102)
(130.0,0.0026798117252673083)
(131.0,0.002472656515338391)
(132.0,0.002281482864721044)
(133.0,0.0021050626109900984)
(134.0,0.0019422612585999171)
(135.0,0.0017920309430101598)
(136.0,0.0016534039069711422)
(137.0,0.0015254864542093671)
(138.0,0.0014074533477106393)
(139.0,0.0012985426217138466)
(140.0,0.0011980507783894887)
(141.0,0.0011053283419735157)
(142.0,0.001019775744852123)
(143.0,0.0009408395217420619)
(144.0,0.0008680087896811664)
(145.0,0.0008008119930338336)
(146.0,0.0007388138941261966)
(147.0,0.0006816127914564873)
(148.0,0.0006288379486793682)
(149.0,0.0005801472187408768)
(150.0,0.0005352248486457551)
(151.0,0.0004937794513741426)
(152.0,0.0004555421324329996)
(153.0,0.00042026475943234763)
(154.0,0.00038771836392072445)
(155.0,0.0003576916655013683)
(156.0,0.0003299897089837897)
(157.0,0.00030443260600767605)
(158.0,0.00028085437321057733)
(159.0,0.00025910185960046944)
(160.0,0.00023903375634192687)
(161.0,0.00022051968267296326)
(162.0,0.0002034393421411952)
(163.0,0.00018768174378531126)
(164.0,0.00017314448329318178)
(165.0,0.00015973307954354413)
(166.0,0.00014736036228608)
(167.0,0.00013594590703681951)
(168.0,0.00012541551356397732)
(169.0,0.00011570072461525146)
(170.0,0.00010673838179289663)
(171.0,9.847021571900823e-05)
(172.0,9.084246785180937e-05)
(173.0,8.380554151562854e-05)
(174.0,7.73136798938862e-05)
(175.0,7.13246689069116e-05)
(176.0,6.579956305582225e-05)
(177.0,6.070243246100021e-05)
(178.0,5.600012945978659e-05)
(179.0,5.166207325373311e-05)
(180.0,4.7660051211871706e-05)
(181.0,4.396803554371383e-05)
(182.0,4.056201415473733e-05)
(183.0,3.7419834588582226e-05)
(184.0,3.452106004463044e-05)
(185.0,3.184683653762048e-05)
(186.0,2.9379770337938252e-05)
(187.0,2.7103814897685793e-05)
(188.0,2.5004166528981034e-05)
(189.0,2.3067168157574648e-05)
(190.0,2.1280220527144387e-05)
(191.0,1.9631700277876724e-05)
(192.0,1.81108843674779e-05)
(193.0,1.6707880343857475e-05)
(194.0,1.541356201665833e-05)
(195.0,1.4219510109813856e-05)
(196.0,1.3117957509617652e-05)
(197.0,1.2101738752603224e-05)
(198.0,1.1164243425041251e-05)
(199.0,1.0299373171248183e-05)
(200.0,9.501502031324714e-06)

 };
\addplot[color=blue] coordinates { 
(0.0,0.0)
(1.0,1.643747700865568e-06)
(2.0,3.6314374998950507e-06)
(3.0,6.035032705479488e-06)
(4.0,8.941551703964949e-06)
(5.0,1.2456216692270612e-05)
(6.0,1.6706260560955718e-05)
(7.0,2.184552931026759e-05)
(8.0,2.806004597363166e-05)
(9.0,3.557473652600552e-05)
(10.0,4.466155987223402e-05)
(11.0,5.564933418332063e-05)
(12.0,6.893561229574682e-05)
(13.0,8.500103165719967e-05)
(14.0,0.0001044266518185947)
(15.0,0.00012791489760299464)
(16.0,0.0001563148521914815)
(17.0,0.0001906527953757372)
(18.0,0.00023216906266503132)
(19.0,0.0002823625159659197)
(20.0,0.0003430441719696282)
(21.0,0.00041640183653600935)
(22.0,0.0005050779489873798)
(23.0,0.0006122632560962014)
(24.0,0.0007418094178945396)
(25.0,0.0008983642009606585)
(26.0,0.0010875335412185096)
(27.0,0.0013160754538203124)
(28.0,0.0015921315198324322)
(29.0,0.0019255024617663133)
(30.0,0.002327975084899451)
(31.0,0.0028137085299580816)
(32.0,0.0033996882311576695)
(33.0,0.004106256014584649)
(34.0,0.004957724131756078)
(35.0,0.0059830793140228725)
(36.0,0.007216779622945768)
(37.0,0.008699641255542014)
(38.0,0.010479803650958614)
(39.0,0.012613748178746241)
(40.0,0.015167327225611265)
(41.0,0.01821673560211573)
(42.0,0.021849324298118533)
(43.0,0.02616411822523938)
(44.0,0.03127185712296793)
(45.0,0.03729433770838102)
(46.0,0.044362804941130346)
(47.0,0.05261513513404387)
(48.0,0.06219159161411316)
(49.0,0.07322903344921965)
(50.0,0.08585363237348545)
(51.0,0.1001724003807713)
(52.0,0.11626412370951922)
(53.0,0.13417058175597996)
(54.0,0.1538891222419005)
(55.0,0.17536768508641173)
(56.0,0.19850316539599674)
(57.0,0.22314358952147165)
(58.0,0.24909402666113672)
(59.0,0.2761256021946197)
(60.0,0.3039865538651695)
(61.0,0.33241407112197807)
(62.0,0.36114570240856186)
(63.0,0.38992935668838424)
(64.0,0.41853127484164765)
(65.0,0.44674170999027035)
(66.0,0.47437836268171635)
(67.0,0.5012878310189604)
(68.0,0.5273454524271886)
(69.0,0.5524539485020591)
(70.0,0.5765412617677984)
(71.0,0.5995579176355142)
(72.0,0.6214741760345359)
(73.0,0.6422771684603512)
(74.0,0.6619681551839246)
(75.0,0.6805599874181872)
(76.0,0.698074820842686)
(77.0,0.7145420989822928)
(78.0,0.7299968057676712)
(79.0,0.7444779742788079)
(80.0,0.7580274314455313)
(81.0,0.7706887548868633)
(82.0,0.7825064169578938)
(83.0,0.7935250915659171)
(84.0,0.803789100783004)
(85.0,0.8133419802753931)
(86.0,0.8222261447901631)
(87.0,0.8304826371909795)
(88.0,0.8381509466967714)
(89.0,0.8452688839813701)
(90.0,0.8518725026047743)
(91.0,0.8579960578569636)
(92.0,0.8636719955063263)
(93.0,0.8689309641684941)
(94.0,0.8738018460636725)
(95.0,0.8783118018292909)
(96.0,0.8824863258180283)
(97.0,0.8863493089562644)
(98.0,0.8899231067807749)
(99.0,0.8932286107264755)
(100.0,0.8962853211180564)
(101.0,0.8991114206346633)
(102.0,0.9017238472791866)
(103.0,0.9041383661005968)
(104.0,0.9063696390963555)
(105.0,0.9084312928683961)
(106.0,0.9103359837257108)
(107.0,0.9120954600236864)
(108.0,0.9137206216086792)
(109.0,0.9152215762991394)
(110.0,0.916607693384468)
(111.0,0.9178876541620203)
(112.0,0.9190694995630677)
(113.0,0.920160674941703)
(114.0,0.9211680721179212)
(115.0,0.9220980687784991)
(116.0,0.9229565653477838)
(117.0,0.9237490194458126)
(118.0,0.9244804780539506)
(119.0,0.9251556075090034)
(120.0,0.9257787214459168)
(121.0,0.9263538068071366)
(122.0,0.9268845480336908)
(123.0,0.9273743495493693)
(124.0,0.9278263566451895)
(125.0,0.9282434748667774)
(126.0,0.9286283880025583)
(127.0,0.9289835747657712)
(128.0,0.9293113242584303)
(129.0,0.929613750300477)
(130.0,0.9298928047025925)
(131.0,0.9301502895564772)
(132.0,0.9303878686118793)
(133.0,0.9306070778053154)
(134.0,0.930809335001266)
(135.0,0.9309959490026504)
(136.0,0.9311681278836217)
(137.0,0.9313269866941423)
(138.0,0.9314735545824214)
(139.0,0.9316087813781171)
(140.0,0.9317335436762058)
(141.0,0.9318486504586125)
(142.0,0.9319548482880566)
(143.0,0.9320528261061047)
(144.0,0.9321432196651088)
(145.0,0.9322266156215645)
(146.0,0.9323035553164056)
(147.0,0.9323745382658847)
(148.0,0.9324400253849514)
(149.0,0.9325004419634108)
(150.0,0.9325561804136416)
(151.0,0.9326076028072612)
(152.0,0.9326550432168093)
(153.0,0.9326988098773357)
(154.0,0.9327391871816404)
(155.0,0.9327764375218975)
(156.0,0.9328108029894094)
(157.0,0.9328425069433741)
(158.0,0.9328717554587028)
(159.0,0.932898738662174)
(160.0,0.932923631965501)
(161.0,0.9329465972032325)
(162.0,0.9329677836828089)
(163.0,0.9329873291535333)
(164.0,0.9330053607006978)
(165.0,0.933021995570633)
(166.0,0.9330373419319999)
(167.0,0.9330514995782462)
(168.0,0.9330645605757567)
(169.0,0.9330766098618931)
(170.0,0.9330877257967896)
(171.0,0.9330979806724686)
(172.0,0.933107441182586)
(173.0,0.933116168855827)
(174.0,0.9331242204557795)
(175.0,0.9331316483498642)
(176.0,0.9331385008497173)
(177.0,0.9331448225252281)
(178.0,0.933150654494279)
(179.0,0.9331560346900508)
(180.0,0.9331609981076431)
(181.0,0.9331655770316011)
(182.0,0.9331698012458294)
(183.0,0.9331736982272547)
(184.0,0.9331772933244914)
(185.0,0.9331806099226774)
(186.0,0.9331836695955468)
(187.0,0.9331864922457237)
(188.0,0.9331890962341556)
(189.0,0.9331914984995215)
(190.0,0.9331937146683948)
(191.0,0.9331957591568735)
(192.0,0.93319764526434)
(193.0,0.9331993852599596)
(194.0,0.9332009904624806)
(195.0,0.9332024713138526)
(196.0,0.9332038374471445)
(197.0,0.9332050977491996)
(198.0,0.9332062604184407)
(199.0,0.933207333018196)
(200.0,0.9332083225258947)

 };
 \addplot[color=black] coordinates {
(0.0,1.0)
(200.0,1.0)

 };
 
  \addplot[color=black,dashed] coordinates {
(0.0,0.06677824875404095 )
(200.0,0.06677824875404095 )

 };

 \legend{ $S$, $I$,$R$, $u_3$,$S_{\infty}$ }
\end{axis} 
\end{tikzpicture} 
\begin{tikzpicture}[thick,scale=0.7, every node/.style={scale=1.0}] \begin{axis}[xlabel=Time (days),
,legend style={at={(0.8,0.8)}}]
\addplot[color=red] coordinates { 
(0.0,0.9999850746268657)
(1.0,0.9999803078410416)
(2.0,0.9999745436707488)
(3.0,0.9999675734463866)
(4.0,0.9999591448501336)
(5.0,0.9999489527900263)
(6.0,0.9999366283679209)
(7.0,0.9999217255440862)
(8.0,0.9999037050188164)
(9.0,0.999881914752167)
(10.0,0.9998555664233847)
(11.0,0.9998237069877883)
(12.0,0.9997851843160165)
(13.0,0.9997386056931209)
(14.0,0.9996822877064497)
(15.0,0.999614195754037)
(16.0,0.9995318710506894)
(17.0,0.999432342587359)
(18.0,0.999312020999904)
(19.0,0.9991665707144549)
(20.0,0.9989907560463698)
(21.0,0.998778256127011)
(22.0,0.998521442608034)
(23.0,0.9982111130418546)
(24.0,0.9978361716627374)
(25.0,0.9973832480125634)
(26.0,0.9968362425075228)
(27.0,0.9961757866992589)
(28.0,0.9953786047695257)
(29.0,0.9944167619087558)
(30.0,0.9932567849702414)
(31.0,0.9918586416166849)
(32.0,0.9901745667425094)
(33.0,0.9881477301913033)
(34.0,0.9857107489634677)
(35.0,0.982784061904316)
(36.0,0.9792742074034005)
(37.0,0.9750720774512845)
(38.0,0.9700512672247988)
(39.0,0.9640667006660917)
(40.0,0.9569537905431716)
(41.0,0.94852848471399)
(42.0,0.938588652087455)
(43.0,0.9269173570719071)
(44.0,0.9132886332698913)
(45.0,0.8974763545798369)
(46.0,0.8792666600235295)
(47.0,0.8584740576134586)
(48.0,0.8349607667482306)
(49.0,0.8086580589887477)
(50.0,0.7795874111061434)
(51.0,0.7478783958711802)
(52.0,0.7137797167218277)
(53.0,0.6776599858855513)
(54.0,0.6399959825037674)
(55.0,0.6013481592647179)
(56.0,0.5623256743894107)
(57.0,0.5235455012337112)
(58.0,0.4855914712077633)
(59.0,0.44897899599784874)
(60.0,0.41412975694395066)
(61.0,0.38135837401688677)
(62.0,0.35087071297606187)
(63.0,0.3447700312861869)
(64.0,0.3447700312861869)
(65.0,0.3447700312861869)
(66.0,0.3447700312861869)
(67.0,0.3447700312861869)
(68.0,0.3447700312861869)
(69.0,0.3447700312861869)
(70.0,0.3447700312861869)
(71.0,0.3447700312861869)
(72.0,0.3447700312861869)
(73.0,0.3447700312861869)
(74.0,0.3447700312861869)
(75.0,0.3447700312861869)
(76.0,0.3447700312861869)
(77.0,0.3447700312861869)
(78.0,0.3447700312861869)
(79.0,0.3447700312861869)
(80.0,0.3447700312861869)
(81.0,0.3447700312861869)
(82.0,0.3447700312861869)
(83.0,0.3447700312861869)
(84.0,0.3447700312861869)
(85.0,0.3447700312861869)
(86.0,0.3447700312861869)
(87.0,0.3447700312861869)
(88.0,0.3447700312861869)
(89.0,0.3447700312861869)
(90.0,0.3447700312861869)
(91.0,0.3447700312861869)
(92.0,0.3447700312861869)
(93.0,0.3447700312861869)
(94.0,0.3447700312861869)
(95.0,0.3447700312861869)
(96.0,0.3447700312861869)
(97.0,0.3447700312861869)
(98.0,0.3447700312861869)
(99.0,0.3447700312861869)
(100.0,0.3447700312861869)
(101.0,0.3447700312861869)
(102.0,0.3447700312861869)
(103.0,0.3447700312861869)
(104.0,0.3447700312861869)
(105.0,0.3447700312861869)
(106.0,0.3447700312861869)
(107.0,0.3447700312861869)
(108.0,0.3447700312861869)
(109.0,0.3447700312861869)
(110.0,0.3447700312861869)
(111.0,0.3447700312861869)
(112.0,0.3447700312861869)
(113.0,0.3447700312861869)
(114.0,0.3447700312861869)
(115.0,0.3447700312861869)
(116.0,0.3447700312861869)
(117.0,0.3447700312861869)
(118.0,0.3447700312861869)
(119.0,0.3447700312861869)
(120.0,0.3447700312861869)
(121.0,0.3447700312861869)
(122.0,0.3447700312861869)
(123.0,0.3447700312861869)
(124.0,0.3447700312861869)
(125.0,0.3447700312861869)
(126.0,0.3447700312861869)
(127.0,0.3447700312861869)
(128.0,0.3447700312861869)
(129.0,0.3447700312861869)
(130.0,0.3447700312861869)
(131.0,0.3447700312861869)
(132.0,0.3447700312861869)
(133.0,0.3447700312861869)
(134.0,0.3447700312861869)
(135.0,0.3447700312861869)
(136.0,0.3447700312861869)
(137.0,0.3447700312861869)
(138.0,0.3447700312861869)
(139.0,0.3447700312861869)
(140.0,0.3447700312861869)
(141.0,0.3447700312861869)
(142.0,0.3447700312861869)
(143.0,0.3447700312861869)
(144.0,0.3447700312861869)
(145.0,0.3447700312861869)
(146.0,0.3447700312861869)
(147.0,0.3447700312861869)
(148.0,0.3447700312861869)
(149.0,0.3447700312861869)
(150.0,0.3447700312861869)
(151.0,0.3447700312861869)
(152.0,0.3447700312861869)
(153.0,0.3447700312861869)
(154.0,0.3447700312861869)
(155.0,0.3447700312861869)
(156.0,0.3447700312861869)
(157.0,0.3447700312861869)
(158.0,0.3447700312861869)
(159.0,0.3447700312861869)
(160.0,0.3447700312861869)
(161.0,0.3447700312861869)
(162.0,0.3447700312861869)
(163.0,0.3447700312861869)
(164.0,0.3447700312861869)
(165.0,0.3447700312861869)
(166.0,0.3447700312861869)
(167.0,0.3447700312861869)
(168.0,0.3447700312861869)
(169.0,0.3447700312861869)
(170.0,0.3447700312861869)
(171.0,0.3447700312861869)
(172.0,0.3447700312861869)
(173.0,0.3447700312861869)
(174.0,0.3447700312861869)
(175.0,0.3447700312861869)
(176.0,0.3447700312861869)
(177.0,0.3447700312861869)
(178.0,0.3447700312861869)
(179.0,0.3447700312861869)
(180.0,0.3447700312861869)
(181.0,0.3447700312861869)
(182.0,0.3447700312861869)
(183.0,0.3447700312861869)
(184.0,0.3447700312861869)
(185.0,0.3447700312861869)
(186.0,0.3447700312861869)
(187.0,0.3447700312861869)
(188.0,0.3447700312861869)
(189.0,0.3447700312861869)
(190.0,0.3447700312861869)
(191.0,0.3447700312861869)
(192.0,0.3447700312861869)
(193.0,0.3447700312861869)
(194.0,0.3447700312861869)
(195.0,0.3447700312861869)
(196.0,0.3447700312861869)
(197.0,0.3447700312861869)
(198.0,0.3447700312861869)
(199.0,0.3447700312861869)
(200.0,0.3447700312861869)

 };
\addplot[color=black!60!green] coordinates { 
(0.0,1.4925373134328359e-05)
(1.0,1.804841125735518e-05)
(2.0,2.182489175077052e-05)
(3.0,2.63915209071741e-05)
(4.0,3.191359816198464e-05)
(5.0,3.859099328096777e-05)
(6.0,4.666537151792117e-05)
(7.0,5.642892660323331e-05)
(8.0,6.823493520935744e-05)
(9.0,8.251051130643704e-05)
(10.0,9.977201674258135e-05)
(11.0,0.00012064367802745847)
(12.0,0.00014588007168688303)
(13.0,0.00017639327522106743)
(14.0,0.00021328564173078715)
(15.0,0.0002578893483591986)
(16.0,0.00031181409711803583)
(17.0,0.0003770046172646669)
(18.0,0.0004558099374304345)
(19.0,0.0005510667695784868)
(20.0,0.0006661997816595564)
(21.0,0.0008053420364521285)
(22.0,0.0009734794429774349)
(23.0,0.001176623702047641)
(24.0,0.0014220189193665533)
(25.0,0.0017183877864745906)
(26.0,0.0020762239512571883)
(27.0,0.002508137846919364)
(28.0,0.0030292637106401894)
(29.0,0.003657735629475884)
(30.0,0.004415239944857443)
(31.0,0.005327649853355174)
(32.0,0.00642574502633137)
(33.0,0.007746013794110007)
(34.0,0.00933152690477447)
(35.0,0.011232858781659737)
(36.0,0.013509012973652259)
(37.0,0.016228281293172213)
(38.0,0.019468929124241516)
(39.0,0.02331955115516111)
(40.0,0.027878882231216005)
(41.0,0.03325477968389297)
(42.0,0.03956202361442529)
(43.0,0.046918524702852446)
(44.0,0.05543950960713969)
(45.0,0.06522930771178088)
(46.0,0.07637053503533946)
(47.0,0.08891080725249678)
(48.0,0.10284764163765556)
(49.0,0.11811290756203223)
(50.0,0.1345589565203705)
(51.0,0.15194920374804796)
(52.0,0.16995615956865226)
(53.0,0.18816943235846784)
(54.0,0.20611489525433138)
(55.0,0.2232841556488694)
(56.0,0.23917116021459137)
(57.0,0.25331090924481636)
(58.0,0.26531450213109914)
(59.0,0.27489540180753075)
(60.0,0.281883689190879)
(61.0,0.28622755486113455)
(62.0,0.2879835846153755)
(63.0,0.26615669181886814)
(64.0,0.24082853381838049)
(65.0,0.217910668729612)
(66.0,0.19717372685579412)
(67.0,0.17841016591272668)
(68.0,0.16143219387584093)
(69.0,0.14606988949449828)
(70.0,0.1321695016630015)
(71.0,0.11959191062785131)
(72.0,0.10821123563049363)
(73.0,0.097913575050378)
(74.0,0.08859586643925539)
(75.0,0.08016485503755567)
(76.0,0.07253616044940986)
(77.0,0.065633432135287)
(78.0,0.059387585270132257)
(79.0,0.053736109319217566)
(80.0,0.04862244241169954)
(81.0,0.043995405250404934)
(82.0,0.03980868889222262)
(83.0,0.036020391272635975)
(84.0,0.03259259783577739)
(85.0,0.029491002072809608)
(86.0,0.026684562170854537)
(87.0,0.024145190336096377)
(88.0,0.021847471681700547)
(89.0,0.019768409867083937)
(90.0,0.017887196942809058)
(91.0,0.016185005097632273)
(92.0,0.014644798223440622)
(93.0,0.0132511614121558)
(94.0,0.01199014667815295)
(95.0,0.010849133362132498)
(96.0,0.00981670181931989)
(97.0,0.00888251912782244)
(98.0,0.00803723567327399)
(99.0,0.007272391574751841)
(100.0,0.00658033201544501)
(101.0,0.005954130630674702)
(102.0,0.00538752018650859)
(103.0,0.004874829855177107)
(104.0,0.004410928459523122)
(105.0,0.003991173118456287)
(106.0,0.0036113627794385644)
(107.0,0.0032676961729384006)
(108.0,0.0029567337680476095)
(109.0,0.0026753633484999685)
(110.0,0.0024207688645647837)
(111.0,0.0021904022490746686)
(112.0,0.001981957915512908)
(113.0,0.0017933496829286562)
(114.0,0.0016226898967367934)
(115.0,0.0014682705364363778)
(116.0,0.0013285461211673827)
(117.0,0.0012021182420187994)
(118.0,0.0010877215662822319)
(119.0,0.0009842111735768576)
(120.0,0.0008905510971014372)
(121.0,0.000805803955330366)
(122.0,0.000729121570384302)
(123.0,0.000659736479180865)
(124.0,0.0005969542524061564)
(125.0,0.0005401465444327805)
(126.0,0.0004887448046256095)
(127.0,0.00044223458709593205)
(128.0,0.00040015040195408714)
(129.0,0.00036207105653019247)
(130.0,0.00032761543993633586)
(131.0,0.000296438708780713)
(132.0,0.0002682288358590576)
(133.0,0.00024270348728150406)
(134.0,0.00021960719678012242)
(135.0,0.00019870880891664354)
(136.0,0.00017979916560113902)
(137.0,0.00016268901276755677)
(138.0,0.00014720710625541728)
(139.0,0.0001331984979406988)
(140.0,0.00012052298496293152)
(141.0,0.00010905370652784719)
(142.0,9.867587424190982e-05)
(143.0,8.928562327149192e-05)
(144.0,8.078897282870924e-05)
(145.0,7.310088558010755e-05)
(146.0,6.61444165644475e-05)
(147.0,5.9849943101670444e-05)
(148.0,5.415446798571529e-05)
(149.0,4.900098898730631e-05)
(150.0,4.433792835648332e-05)
(151.0,4.011861661514427e-05)
(152.0,3.630082547322219e-05)
(153.0,3.2846345193764784e-05)
(154.0,2.9720602177044387e-05)
(155.0,2.689231293635113e-05)
(156.0,2.4333171002343307e-05)
(157.0,2.201756362238808e-05)
(158.0,1.9922315419524358e-05)
(159.0,1.802645644550064e-05)
(160.0,1.6311012306484706e-05)
(161.0,1.4758814260952587e-05)
(162.0,1.3354327389152817e-05)
(163.0,1.2083495114408088e-05)
(164.0,1.0933598520171289e-05)
(165.0,9.893129054833714e-06)
(166.0,8.951673350272395e-06)
(167.0,8.099809001361891e-06)
(168.0,7.329010263376818e-06)
(169.0,6.631562723473018e-06)
(170.0,6.00048609225092e-06)
(171.0,5.4294643426730785e-06)
(172.0,4.912782497142684e-06)
(173.0,4.4452694300869e-06)
(174.0,4.022246113594068e-06)
(175.0,3.6394787881296784e-06)
(176.0,3.2931365896479446e-06)
(177.0,2.9797532090168316e-06)
(178.0,2.6961922000311913e-06)
(179.0,2.4396155888049482e-06)
(180.0,2.207455470374574e-06)
(181.0,1.997388308243111e-06)
(182.0,1.8073116796459348e-06)
(183.0,1.635323233797083e-06)
(184.0,1.4797016525231886e-06)
(185.0,1.3388894227326424e-06)
(186.0,1.211477248301077e-06)
(187.0,1.0961899453620702e-06)
(188.0,9.91873679838409e-07)
(189.0,8.974844194828239e-07)
(190.0,8.120774848523512e-07)
(191.0,7.347980946389471e-07)
(192.0,6.648728107308565e-07)
(193.0,6.016017973840273e-07)
(194.0,5.443518170307631e-07)
(195.0,4.925498926253056e-07)
(196.0,4.456775730969764e-07)
(197.0,4.032657445176063e-07)
(198.0,3.648899350516645e-07)
(199.0,3.3016606669946156e-07)
(200.0,2.987466113154333e-07)

 };
\addplot[color=blue] coordinates { 
(0.0,0.0)
(1.0,1.643747700865568e-06)
(2.0,3.6314374998950507e-06)
(3.0,6.035032705479488e-06)
(4.0,8.941551703964949e-06)
(5.0,1.2456216692270612e-05)
(6.0,1.6706260560955718e-05)
(7.0,2.184552931026759e-05)
(8.0,2.806004597363166e-05)
(9.0,3.557473652600552e-05)
(10.0,4.466155987223402e-05)
(11.0,5.564933418332063e-05)
(12.0,6.893561229574682e-05)
(13.0,8.500103165719967e-05)
(14.0,0.0001044266518185947)
(15.0,0.00012791489760299464)
(16.0,0.0001563148521914815)
(17.0,0.0001906527953757372)
(18.0,0.00023216906266503132)
(19.0,0.0002823625159659197)
(20.0,0.0003430441719696282)
(21.0,0.00041640183653600935)
(22.0,0.0005050779489873798)
(23.0,0.0006122632560962014)
(24.0,0.0007418094178945396)
(25.0,0.0008983642009606585)
(26.0,0.0010875335412185096)
(27.0,0.0013160754538203124)
(28.0,0.0015921315198324322)
(29.0,0.0019255024617663133)
(30.0,0.002327975084899451)
(31.0,0.0028137085299580816)
(32.0,0.0033996882311576695)
(33.0,0.004106256014584649)
(34.0,0.004957724131756078)
(35.0,0.0059830793140228725)
(36.0,0.007216779622945768)
(37.0,0.008699641255542014)
(38.0,0.010479803650958614)
(39.0,0.012613748178746241)
(40.0,0.015167327225611265)
(41.0,0.01821673560211573)
(42.0,0.021849324298118533)
(43.0,0.02616411822523938)
(44.0,0.03127185712296793)
(45.0,0.03729433770838102)
(46.0,0.044362804941130346)
(47.0,0.05261513513404387)
(48.0,0.06219159161411316)
(49.0,0.07322903344921965)
(50.0,0.08585363237348545)
(51.0,0.1001724003807713)
(52.0,0.11626412370951922)
(53.0,0.13417058175597996)
(54.0,0.1538891222419005)
(55.0,0.17536768508641173)
(56.0,0.19850316539599674)
(57.0,0.22314358952147165)
(58.0,0.24909402666113672)
(59.0,0.2761256021946197)
(60.0,0.3039865538651695)
(61.0,0.33241407112197807)
(62.0,0.36114570240856186)
(63.0,0.38907327689494414)
(64.0,0.4144014348954318)
(65.0,0.4373192999842001)
(66.0,0.4580562418580181)
(67.0,0.47681980280108555)
(68.0,0.4937977748379715)
(69.0,0.5091600792193143)
(70.0,0.5230604670508111)
(71.0,0.5356380580859611)
(72.0,0.5470187330833189)
(73.0,0.5573163936634344)
(74.0,0.5666341022745573)
(75.0,0.5750651136762572)
(76.0,0.5826938082644031)
(77.0,0.5895965365785257)
(78.0,0.5958423834436802)
(79.0,0.6014938593945952)
(80.0,0.6066075263021129)
(81.0,0.6112345634634078)
(82.0,0.6154212798215899)
(83.0,0.6192095774411766)
(84.0,0.6226373708780351)
(85.0,0.6257389666410029)
(86.0,0.6285454065429587)
(87.0,0.6310847783777164)
(88.0,0.6333824970321124)
(89.0,0.6354615588467292)
(90.0,0.6373427717710041)
(91.0,0.6390449636161809)
(92.0,0.6405851704903729)
(93.0,0.6419788073016578)
(94.0,0.6432398220356607)
(95.0,0.6443808353516814)
(96.0,0.6454132668944942)
(97.0,0.6463474495859914)
(98.0,0.64719273304054)
(99.0,0.6479575771390619)
(100.0,0.6486496366983688)
(101.0,0.6492758380831385)
(102.0,0.6498424485273047)
(103.0,0.6503551388586358)
(104.0,0.65081904025429)
(105.0,0.6512387955953569)
(106.0,0.6516186059343743)
(107.0,0.6519622725408745)
(108.0,0.652273234945765)
(109.0,0.6525546053653127)
(110.0,0.6528091998492479)
(111.0,0.6530395664647384)
(112.0,0.6532480107983)
(113.0,0.6534366190308845)
(114.0,0.6536072788170763)
(115.0,0.6537616981773771)
(116.0,0.6539014225926465)
(117.0,0.6540278504717951)
(118.0,0.6541422471475317)
(119.0,0.6542457575402372)
(120.0,0.6543394176167128)
(121.0,0.6544241647584838)
(122.0,0.6545008471434299)
(123.0,0.6545702322346337)
(124.0,0.6546330144614084)
(125.0,0.654689822169382)
(126.0,0.6547412239091891)
(127.0,0.6547877341267189)
(128.0,0.6548298183118607)
(129.0,0.6548678976572846)
(130.0,0.6549023532738785)
(131.0,0.6549335300050341)
(132.0,0.6549617398779557)
(133.0,0.6549872652265332)
(134.0,0.6550103615170348)
(135.0,0.6550312599048984)
(136.0,0.655050169548214)
(137.0,0.6550672797010475)
(138.0,0.6550827616075593)
(139.0,0.6550967702158738)
(140.0,0.6551094457288514)
(141.0,0.6551209150072865)
(142.0,0.6551312928395725)
(143.0,0.6551406830905431)
(144.0,0.6551491797409857)
(145.0,0.6551568678282343)
(146.0,0.6551638242972501)
(147.0,0.6551701187707129)
(148.0,0.655175814245829)
(149.0,0.6551809677248275)
(150.0,0.6551856307854581)
(151.0,0.6551898500971995)
(152.0,0.6551936678883418)
(153.0,0.6551971223686208)
(154.0,0.6552002481116372)
(155.0,0.6552030764008782)
(156.0,0.6552056355428124)
(157.0,0.6552079511501923)
(158.0,0.6552100463983952)
(159.0,0.655211942257369)
(160.0,0.6552136577015083)
(161.0,0.6552152098995535)
(162.0,0.6552166143864255)
(163.0,0.6552178852187006)
(164.0,0.6552190351152947)
(165.0,0.65522007558476)
(166.0,0.6552210170404644)
(167.0,0.6552218689048132)
(168.0,0.655222639703551)
(169.0,0.6552233371510914)
(170.0,0.6552239682277228)
(171.0,0.6552245392494725)
(172.0,0.6552250559313182)
(173.0,0.6552255234443851)
(174.0,0.655225946467701)
(175.0,0.6552263292350264)
(176.0,0.6552266755772247)
(177.0,0.655226988960605)
(178.0,0.6552272725216138)
(179.0,0.6552275290982251)
(180.0,0.6552277612583436)
(181.0,0.6552279713255058)
(182.0,0.6552281614021341)
(183.0,0.6552283333905798)
(184.0,0.6552284890121611)
(185.0,0.655228629824391)
(186.0,0.6552287572365652)
(187.0,0.6552288725238679)
(188.0,0.6552289768401334)
(189.0,0.6552290712293938)
(190.0,0.6552291566363283)
(191.0,0.6552292339157185)
(192.0,0.6552293038410024)
(193.0,0.6552293671120157)
(194.0,0.6552294243619964)
(195.0,0.6552294761639207)
(196.0,0.6552295230362402)
(197.0,0.6552295654480687)
(198.0,0.6552296038238781)
(199.0,0.6552296385477466)
(200.0,0.6552296699672023)

 };
\addplot[color=black] coordinates {
(0.0,1.0)
(62.0,1.0)
(63.0,0.0)
(200.0,0.0)
};

\addplot[color=black,dashed] coordinates {
(0.0,0.3448275862068966)
(200.0,0.3448275862068966)

 };
 
  \legend{ $S$, $I$,$R$, $u_3$,$S_\herd$ }
\end{axis} 
\end{tikzpicture} 

\end{center}
  \caption{Numerical simulation of the SIR model \eqref{eq11} with the numerical parameters 
 given in Table \ref{tab:value} and $u_1(\cdot)=u_2(\cdot)=0$. Left: no action ($u_3(\cdot)=1$).
Right: bang-bang action with one switch at the epidemic peak ($u_3(\cdot)=1$ on $(0,62)$ and $u_3(\cdot)=0$ otherwise).\label{fig1}}
\end{figure}

For better readability, all results are exposed in Section \ref{se2}.
It is shown in Section~\ref{se20} that, at least for strong enough lockdown measures and long enough intervention time $T$, it is possible to stop the epidemics arbitrarily close past the herd immunity.
In Section \ref{se22}, we provide and analyze the optimal control law that leads asymptotically, through an intervention of duration $T$, to the largest number of susceptible individuals.
Last, we show in Section \ref{se23} that this optimal strategy coincides with a time minimal policy.
For the sake of readability, all the proofs are postponed to Section \ref{se3}.
These results are numerically illustrated in Section \ref{se24}.


\section{Main results}
\label{se2}

For the sake of clarity, the demonstrations of all the results of this section are postponed to Section~\ref{se3}.

According to the introduction in the previous section, we consider in the sequel the following ``SIR type'' system
\begin{equation}\label{SIR}
\begin{array}{ll}
\dot S(t) = -u(t)\beta S(t)I(t), & t\geqslant 0\\
\dot I(t) = u(t) \beta S(t)I(t) - \nu I(t), &t\geqslant 0
\end{array}
\end{equation}
complemented with nonnegative initial data $S(0)=S_0$, $I(0)=I_0$ such that $S_0+I_0 \leq 1$.
For given $T>0$ and $\alpha\in[0,1)$, the input control $u$ will be assumed to belong to the so-called {\it admissible set} $\mathcal{U}_{\alpha,T} $ defined by
\begin{equation*}
  \mathcal{U}_{\alpha,T} := \{u \in L^\infty([0,+\infty)),\ \alpha\leq u(t) \leq 1 \text{ if } t\in [0,T], u(t)=1 \text{ if } t> T\}.
\end{equation*}
The constant $T$  characterizes the duration of the intervention, and $\alpha$ its intensity  (typically the strength of a lockdown procedure).

\subsection{Toward an optimal control problem: reachable asymptotic immunity levels}
\label{se20}

The following result assesses the question of stopping the evolution exactly at, or arbitrarily close to, the herd immunity $S_\herd$ defined by \eqref{eq33}.

\begin{proposition}
\label{prop:Sinf}
Let $\alpha\in [0,1)$ and $T>0$.  Assume that $S_0 > S_\herd$ and consider
  \begin{equation}\label{eq:alpha}
   \overline{\alpha}:=\frac{S_\herd}{S_0+I_0-S_\herd}(\ln S_0 - \ln S_\herd).
  \end{equation}
  \begin{enumerate}
  \item[(i)] If $\alpha\leq \overline{\alpha}$, then, for all $\eps\in (0,S_\herd)$, there exist $T>0$ and a control $u\in \cU_{\alpha,T}$ such that the solution $S$ to \eqref{SIR} associated to $u$ satisfies
  \[
S_\herd\geq  \lim_{t\to + \infty} S(t) \geq S_\herd - \eps.
  \]
    \item[(ii)] If $\alpha> \overline{\alpha}$, then for all $u\in \cU_{\alpha,T}$  the solution $S$ to \eqref{SIR}  associated to $u$ satisfies
    $$  \lim_{t\to + \infty} S(t)\leq  \lim_{t\to + \infty} S^{\alpha}(t)<S_\herd, $$
    where $S^{\alpha}$ is the solution to 
    \eqref{SIR} associated to  $u\equiv\alpha$.

 Moreover, the map $\alpha\mapsto  \lim_{t\to + \infty} S^{\alpha}(t)$ is strictly decreasing.
  \item[(iii)] There is no $T\in (0,+\infty)$ and $u\in \cU_{\alpha,T}$ such that the solution $S$ to \eqref{SIR}  associated to $u$ satisfies
   \[
  \lim_{t\to + \infty} S(t) = S_\herd.
  \]
  \end{enumerate}
\end{proposition}

The proof of Proposition \ref{prop:Sinf} is given in Section \ref{sec:proof prop}.

\begin{remark}
The values given in Table~\ref{tab:value}, which correspond to the sanitary measures put in place in France during the March to May 2020 containment period, satisfy $S_0 > S_\herd$ and $\alpha\leq \overline{\alpha}$.
According to Table~\ref{fig1}, there thus exists a containment strategy that increases the limit proportion of susceptible by $30\%$.
\end{remark}

\noindent {\bf Conclusion of this section.} To tackle the issue of reaching the closest state to the threshold of herd immunity with the help of control functions $u(\cdot)$, it seems relevant to maximize the limit  of $S(t)$ when $t\to +\infty$ for a given duration $T$ of the intervention.
This leads to consider the optimal control problem
\begin{equation}\tag{$\mathcal{P}_{\alpha,T}$}
\label{OCP}
\boxed{\sup_{u\in\cU_{\alpha, T}}S_\infty(u)}
\end{equation}
where $$S_\infty(u):=\lim\limits_{t\rightarrow\infty}S(t),$$ with $(S,I)$ the solution to \eqref{SIR} associated to $u$.

We now make an observation, which will be crucial in the analysis (see details in Section \ref{se31}).
It turns out that the quantity $S+I-\frac{1}{\mathcal{R}_0}\ln(S)$ is constant on any time interval on which $u(\cdot)= 1$ (see Lemma~\ref{co1}).
Therefore, using the fact that $\lim\limits_{t\rightarrow +\infty}I(t) =0$ and the monotonicity of $x\mapsto x- \frac{1}{\mathcal{R}_0}\ln x$, 
 the optimal control problem~\eqref{OCP} is indeed equivalent to
\[
\boxed{\inf_{u\in\cU_{\alpha, T}}S(T)+I(T)-\frac{1}{\mathcal{R}_0}\ln(S(T) )}
\]
where $(S,I)$ is the solution to \eqref{SIR} associated to $u$.

\subsection{Optimal immunity control}
\label{se22}
This section is devoted to the analysis results of the optimal control problem \eqref{OCP}.

The first result of this section reduces the study of the complete problem to that of a one dimensional optimization problem, whose unknown denoted $T_0$ stands for a switching time.
For simplicity, given $\alpha\in [0,1]$, $T>0$ and $T_0\in [0,T]$, we define the function  $u_{T_0}\in \cU_{\alpha,T}$ by
\begin{equation}
\label{eq55}
u_{T_0}=\mathds{1}_{[0,T_0]}+\alpha \mathds{1}_{[T_0,T]}+\mathds{1}_{[T,+\infty)}.
\end{equation}
Also, we denote $(S^{T_0}, I^{T_0})$ the solution of \eqref{SIR} with $u=u_{T_0}$.

\begin{theorem}\label{th1}
Let $\alpha\in [0,1)$ and $T>0$. Problem~\eqref{OCP} admits a unique solution $u^*$. Furthermore,
  \begin{enumerate}
  \item[(i)]
   the maximal value 
  $
 S_{\infty,\alpha,T}^*:=\max \{ S_\infty(u): \ u\in \mathcal{U}_{\alpha,T}\} $ is nonincreasing with respect to $\alpha$ and nondecreasing with respect to $T$.
\item[(ii)] there exists a unique $T_0\in[0, T)$ such that $u^*=u_{T_0}$ (in particular, the optimal control is {\it bang-bang}).
\end{enumerate}
\end{theorem}
The proof of Theorem \ref{th1} is given in Section \ref{se31}.

\begin{tabular}{ll}
\begin{minipage}{3.3cm}
\begin{center}
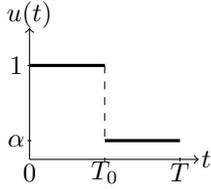

\begin{tikzpicture}[scale=0.25]
\draw[->](0,0) -- (9,0);
\draw[->](0,0) -- (0,7);
\draw[very thick] (0,5) -- (4,5);
\draw[dashed](4,5)--(4,1);
\draw[very thick] (4,1) -- (8,1);
\node at (4,-0.7) {$T_0$};
\node at (0,-0.7) {$0$};
\node at (8,-0.7) {$T$};
\node at (-0.7,1) {$\alpha$};
\node at (-0.7,5) {$1$};
\node at (9.4,0) {$t$};
\node at (0,7.7) {$u(t)$};
\draw(4,-0.1) -- (4,0.1);
\draw(8,-0.1) -- (8,0.1);
\draw(-0.1,5) -- (0.1,5);
\draw(-0.1,1) -- (0.1,1);
\end{tikzpicture}
\captionof{figure}{Opt. control}
\end{center}
\end{minipage} & 
\begin{minipage}{8.2cm}
At this step, the result above proves that the optimal control belongs to a family of {\it bang-bang} functions parametrized by the switching time $T_0$. This reduces the study to a rather simple one-dimensional optimization problem, at least from the numerical point of view. 
\end{minipage}
\end{tabular}

We deduce that the switching time associated to the optimal solution of Problem~\eqref{OCP} solves the 1D optimization problem
\begin{equation}
\boxed{\sup_{T_0\in [0,T)}j(T_0)
\tag{$\widetilde\cP_{\alpha,T}$}}
\label{eq:opt T0}
\end{equation}
where
$j(T_0):=S_\infty(u_{T_0}),$
 and $u_{T_0}$ is defined by \eqref{eq55}.
 The  next result characterizes the optimal  $T_0$ for \eqref{eq:opt T0} whenever $\alpha\in [0,1]$.
To state it, one needs to introduce the function $\psi$ given by
\begin{equation}
\label{eq:psi}
\psi : T_0\ni [0,T]\mapsto (1-\alpha) \beta I^{T_0}(T) \int_{T_0}^T \frac{S^{T_0}(t)}{I^{T_0}(t)}\,dt - 1,
\end{equation} 
where $(S^{T_0}, I^{T_0})$ denotes the solution to \eqref{SIR} with $u=u_{T_0}$.

\begin{theorem}\label{th:T0}
Let $T>0$, $\alpha \in [0,1)$ and $T_0^*$ the unique optimal solution to Problem~\eqref{eq:opt T0} (hence $u_{T_0^*}$ is the optimal solution to Problem~\eqref{OCP} with the notation \eqref{eq55}).
One has the following characterization:
\begin{itemize}
\item if $\psi(0)\leq0$, then $T_0^*=0$.
\item if $\psi(0)>0$, then $T_0^*$ is the unique  solution on $(0,T)$ to the equation
  \begin{equation}\label{eq:formula}
\psi(T_0)=0.
  \end{equation} 
  \end{itemize}
 Moreover $S^T(T_0^*)\geqslant S_\herd$, i.e. $T_0^*\leqslant (S^T)^{-1}(S_\herd)$, where,  in agreement with \eqref{eq55}, $S^T$ denotes the solution to System~\eqref{SIR} with $u=u_T\equiv 1$.

In the particular case $\alpha=0$, one has $T_0^*>0$ if, and only if $T> \frac{1}{\nu} \ln \frac{S_0}{S_0-S_\herd}$, and in that case, $T_0^*$ is the unique solution to the equation
\begin{equation*}
 S^{T_0}(T_0) = \frac{S_\herd }{1-e^{\nu (T_0-T)}}.
\end{equation*}
Moreover the functional $j$ is increasing on $(0,T_0^*)$, then decreasing on $(T_0^*,T)$ whatever the value of $\alpha\in [0,1]$ be.
\end{theorem}

The reduction to the optimal control \eqref{eq:opt T0} and this  characterization of the optimal $T_0^*$ constitute useful tools to refine the numerical search of optimal controls.

\noindent {\bf Ongoing work.} We conclude this part by underlying that the optimal control problem we have solved is a prototypal one.
In particular, it does not appear completely relevant since it does not take into account a cost related to the occupancy of hospital beds and the fact that this number of beds is limited. This leads, as a first approximation, to add to the problem of optimal control a constraint on the total number of infected persons. This will be the main topic of future work. 
 
\subsection{Relations with the minimal time problem}
\label{se23}
We have shown (see Theorem~\ref{prop:Sinf}) that, provided that $\alpha$ is
 sufficiently small, there exists for every $\eps>0$ a time of control $T$ such that, by acting on $[0,T]$, we have  
$\lim_{t\to +\infty} S(t) \geq S_\herd-\eps$, 
i.e.\ at a distance $\eps$ of the herd immunity threshold.
Then, we may wonder what is the minimal time of action such that the final value of susceptible is at a given distance of this threshold.
This amounts to solve the following optimal control problem.\\
\begin{center}
  \fbox{\begin{minipage}{12.5cm}
    \textbf{Minimal time problem: }for $\eps>0$, determine the minimal time of action $T^*>0$ such that
     the optimal final number of susceptible individuals satisfies $$S^*_{\infty,\alpha,T^*} \geqslant S_\herd-\eps.$$
  \end{minipage}
}\end{center}\ 

We recall that $S^*_{\infty,\alpha,T}$ is defined in Theorem \ref{th1}. The following result answers to this question by noting that solving this problem is equivalent to solve Problem~\eqref{OCP}.
\begin{theorem}\label{theo:min time}
  Assume $\alpha \leqslant \overline{\alpha}$ (defined in \eqref{eq:alpha}) and let $\eps>0$.

Let $T^*_\eps>0$ be the solution to the minimal time problem above and denote $u^*_\eps$ the corresponding control function. Then, $u^*_\eps$ is the unique solution of Problem~\eqref{OCP} determined in Theorems~\ref{th1} and \ref{th:T0} associated to $T=T^*$.

Conversely, let $T>0$ and $S_{\infty,\alpha,T}^*$ the maximum of Problem~\eqref{OCP}. Then, $T$ is the minimal time of intervention such that $S_{\infty}(u)\geqslant S_{\infty,\alpha,T}^*$ for some $u\in\mathcal{U}_{\alpha,T}$.
\end{theorem}

\section{Proofs of the main results}
\label{se3}

\subsection{Preliminary results}
Before proving the main results,
 we will provide some useful elementary properties of the state variables $(S,I)$ solving \eqref{SIR} whose role in the sequel will be central. To this aim, it is convenient to introduce the function $\Phi_\gamma $ defined for any $\gamma>0$ 
\begin{equation}
\label{eq40}
\Phi_\gamma :\RR_+^*\times \RR_+\ni (S,I) \longmapsto S + I - \frac{1}{\gamma} \ln S.
\end{equation}

Let us start with a preliminary result regarding the value of $\Phi_\gamma$ along the trajectories of System~\eqref{SIR}.
\begin{lemma}
\label{le1}
For any $u\in L^\infty([0,+\infty),[0,1])$ and $\gamma\in\RR$, one has
\begin{equation}
\label{eq34}
\frac{d}{dt} \left[
\Phi_\gamma(S(t),I(t))
\right]
= \left(
\frac{\beta}{\gamma} u(t) - \nu
\right) I
\end{equation}
along any trajectory of system \eqref{SIR}.
In particular, if $u$ is constant on a non-empty, possibly unbounded, interval, then the function $\Phi_{\cR_0 u}(S(\cdot),I(\cdot))$ is constant on this interval along any trajectory of System~\eqref{SIR}.
\end{lemma}

\begin{proof}
The proof of \eqref{eq34} follows from straightforward computations. Indeed along every trajectory of \eqref{SIR}, one has
\begin{eqnarray*}
\frac{d}{dt} \left[
\Phi_\gamma(S(t),I(t))
\right]
& = &
\frac{d}{dt}\left(
S + I -\frac{1}{\gamma} \ln S
\right)
= \left(
1- \frac{1}{\gamma S}
\right)\dot S + \dot I
= \left(
\frac{\beta}{\gamma} u(t) - \nu
\right) I.
\end{eqnarray*}
The second part of the statement is an obvious byproduct of this property, by setting $\gamma=\mathcal R_0u$.
\end{proof}

Lemma \ref{le1} allows to characterize the value of the limit of $S$ at infinity, as now stated.

\begin{lemma}
\label{co1}
Let $u\in\cU_{\alpha, T}$. For any trajectory of \eqref{SIR}, the limit $S_\infty(u)$ of $S(t)$ at infinity exists and is the unique solution in the interval $[0,1/\cR_0]$ of the equation
\begin{equation}
\label{eq38}
\Phi_{\cR_0}(S_\infty,0) = \Phi_{\cR_0}(S(T),I(T)),
\end{equation}
where $\Phi_{\cR_0}$ is given by \eqref{eq40}.
\end{lemma}

\begin{proof}
Any input control $u$ from $\cU_{\alpha,T}$ is equal to 1 on $[T,+\infty)$.
Hence, applying Lemma~\ref{le1} with $u=1$ on this interval yields \eqref{eq38}, by continuity of $\Phi_{\cR_0}$ and because of \eqref{eq41}.
Moreover, Eq.~\eqref{eq38} has exactly two roots. Indeed, this follows by observing that the mapping $S\mapsto \Phi_{\cR_0}(S,0)$
is first decreasing and then increasing on $(0,+\infty)$, with infinite limit at $0^+$ and $+\infty$, and minimal value at $S=1/\cR_0=S_\herd$, equal to $\frac{1}{\cR_0} (1+\ln \cR_0)$. We conclude by noting that the limit $S_\infty(u)$ cannot be larger than $S_\herd$: otherwise there would exist $\eps>0$ such that $\dot I>\eps>0$ and $\dot S<-\eps\beta S$ for $T$ large enough, so that $S$ would tend to zero at infinity, yielding a contradiction.
It follows that the value of $S_\infty(u)$ is thus the smallest root of \eqref{eq38}.
\end{proof}

\begin{remark}[On the control in infinite time]
Notice that, in the quite unrealistic situation where one is able to act on System~\eqref{SIR} up to an infinite horizon of time, then the optimal strategy to maximize $S_\infty(u)$ is to consider the constant control function $u_\alpha(\cdot)=\alpha$ on $[0,+\infty)$.
Indeed, according to Lemma~\ref{le1},
  \[
  \frac{d}{dt}[\Phi_{\alpha\cR_0}(S(t),I(t))]
  = \left(
  \frac{\beta}{\alpha\cR_0} u(t) -\nu
  \right) I(t)
  = \frac{\nu}{\alpha} (u(t)-\alpha)I(t) \geq 0,
  \]
 for all $t\in [0,T]$.
  Hence,
  $
 \Phi_{\alpha\cR_0}(S_0,I_0) \leq \Phi_{\alpha\cR_0}(S_\infty(u),0),
  $ with equality if, and only if $u=u_\alpha=\alpha$ a.e.\ on $\RR_+$.
  Since $S_\infty(u)$ is maximal whenever $\Phi_{\alpha\cR_0}(S_\infty(u),0)$
  is minimal, we get that the optimal strategy in that case corresponds to the choice $u=u_{\alpha}=\alpha$ a.e.\ on $\RR_+$.
 Moreover, it is notable that this maximal value $S_\infty(u_\alpha)$ is computed by solving the nonlinear equation $\Phi_{\alpha\cR_0}(S_\infty,0) = \Phi_{\alpha\cR_0}(S_0,I_0)$. Therefore, an easy application of the implicit functions theorem
yields that the mapping $\alpha\mapsto S_\infty(u_\alpha)$ is decreasing.
\end{remark}

\subsection{Proof of Proposition~\ref{prop:Sinf}}\label{sec:proof prop}
 Let us start with $(iii)$.
 According to \eqref{SIR}, one has $\dot I\geq -\nu I$, thus $I(t)\geq I_0 e^{-\nu t}>0$ for all $t>0$. Then, $\dot S<0$ and $S$ is decreasing.
 Assume by contradiction that, for all $t>0$, we have $S_\herd < S(t)$.
 Since $I$ satisfies \eqref{SIR}, it follows that $\dot I>0$ on $(T,+\infty)$ and then, $\dot S\leq -\beta I(T) S$ on $(T,+\infty)$.
 We thus infer that $S(t)\leq e^{- \beta I(T) (t-T)} S(T)$ for $t>T$, and thus $S(t) \to 0$ as $t\to +\infty$, which is a contradiction.

Let us now show $(i)$. For all $\alpha\in[0,1)$, we denote by 
$u_\alpha(\cdot)$ the control equal to $\alpha$  on $(0,\infty)$.
  Thanks to the same argument by contradiction as above, one has $S_{\infty}({u_{\overline{\alpha}}})\leqslant 1/(\overline{\alpha}\mathcal{R}_0)$. 
Let us denote by $(S^{u_{\overline{\alpha}}},I^{u_{\overline{\alpha}}})$ the solution to System~\eqref{SIR} associated to $u_{\overline{\alpha}}$.
  Lemma~\ref{le1} shows that the function $t\mapsto \Phi_{\overline{\alpha}\cR_0}(S^{u_{\overline{\alpha}}}(t),I^{u_{\overline{\alpha}}}(t))$
 is conserved, and we infer that $S_\infty({u_{\overline{\alpha}}})$ solves the equation
  \[
    \Phi_{\overline{\alpha}\cR_0} (S_0,I_0)=  \Phi_{\overline{\alpha}\cR_0} (S_\infty({u_{\overline{\alpha}}}),0).
  \]
  Using the expression of $\overline{\alpha}$, 
  \[
    \Phi_{\overline{\alpha}\cR_0} (S_\herd,0)= \Phi_{\overline{\alpha}\cR_0} (S_0,I_0)=\Phi_{\overline{\alpha}\cR_0} (S_\infty({u_{\overline{\alpha}}}),0).
  \]
Since  $x\mapsto \Phi_{\overline{\alpha}\cR_0} (x,0)$ is bijective on $(0,\frac{1}{\overline{\alpha} \cR_0})$, we deduce that $S_\infty({u_{\overline{\alpha}}})=S_\herd$.
It follows that for $\eta>0$ small enough, there exists $T>0$ such that for each $t>T$, $\dot I^{u_{\overline{\alpha}}}(t)\leqslant (\beta \overline{\alpha} (1+\eta)\frac{\nu}{\beta}-\nu)I^{u_{\overline{\alpha}}}(t)<0$. By using a Gronwall lemma, one infers that $I^{u_{\overline{\alpha}}}(t) \rightarrow 0$ as $t\to+\infty$. 
Then, for all $k\in \mathbb{N}^*$, there exists $T_k>0$, such that $|S^{u_{\overline{\alpha}}}(T_k)-S_\herd|\leqslant 1/k$ and $I^{u_{\overline{\alpha}}}(T_k)\leqslant 1/k$. Consider $u_k:=\overline{\alpha}\mathds{1}_{(0,T_k)}+\mathds{1}_{(T_k,\infty)}$
and let us denote by $(S^{u_{k}},I^{u_{k}})$ the solution to System~\eqref{SIR} associated to $u_{k}$.
 By continuity of $\Phi_{\cR_0}$,
  \[
    \Phi_{\cR_0} (S_\infty(u_k),0) = \Phi_{\cR_0} (S^{u_k}(T_{k}),I^{u_k}(T_{k}))\underset{k\rightarrow\infty}{\longrightarrow}\Phi_{\cR_0} (S_\herd,0).
  \]
Thus $S_\infty(u_k)\underset{k\rightarrow\infty}{\longrightarrow}S_\herd$.

  Let us finally prove $(ii)$. We show that $\alpha\mapsto S_\infty(u_\alpha)$ is strictly decreasing.  Let $\alpha_1,\alpha_2\in[0,1)$ such that $\alpha_1<\alpha_2$ 
 and  $(S^{u_{\alpha_1}},I^{u_{\alpha_1}})$ and $(S^{u_{\alpha_2}},I^{u_{\alpha_2}})$ the solutions to System~\eqref{SIR} associated to $u_{\alpha_1}$ and $u_{\alpha_2}$, respectively.
 Using Lemma \ref{le1},  
 $t\mapsto \Phi_{ \alpha_1\cR_0} (S^{u_{\alpha_2}}(t),I^{u_{\alpha_2}}(t))$ is strictly increasing, hence
 $$\Phi_{{ \alpha_1}\cR_0}(S_{\infty}(u_{\alpha_1}),0)=\Phi_{{ \alpha_1}\cR_0}(S_0,I_0)<\Phi_{{ \alpha_1}\cR_0}(S_{\infty}(u_{\alpha_2}),0).$$
 Thanks to the equations satisfied by $(S^{u_{\alpha_1}},I^{u_{\alpha_1}})$ and $(S^{u_{\alpha_2}},I^{u_{\alpha_2}})$, one has $S_{\infty}(u_{\alpha_1})$, $S_{\infty}(u_{\alpha_2})\leqslant S_\herd/\alpha_1$. Since 
 $x\mapsto \Phi_{{ \alpha_1}\cR_0}(x,0)$ in strictly decreasing on $(0,S_\herd/\alpha_1)$, we deduce that $S_{\infty}(u_{\alpha_2})<S_{\infty}(u_{\alpha_1})$.
This concludes the proof since $S_\infty(u_{\overline{\alpha}})=S_\herd$.


\subsection{Proof of Theorem~\ref{th1}}
\label{se31}

Solving \eqref{OCP} involves the resolution of an ODE system on an infinite horizon, and it is quite convenient to consider an equivalent version of this problem involving an ODE system on a bounded horizon.
A key point for this is that, according to Lemma~\ref{co1}, $S_\infty$ solves Eq.~\eqref{eq38}.
Furthermore, since the mapping $[0,1/\mathcal R_0]\ni S\mapsto \Phi_{\cR_0}(S,0) $ is decreasing, maximizing $S_\infty$ is equivalent to minimize $ \Phi_{\cR_0}(S_\infty,0) $.
Combining all these observations yields that the optimal control problem is equivalent to the following version we will investigate hereafter:
\begin{equation}\tag{$\mathcal{P}_{\alpha,T}^{\Phi}$}
\label{OCP2}
\inf_{u\in\cU_{\alpha, T}} J_\Phi(u),
\end{equation}
where 
\begin{equation*}
J_\Phi(u):=\Phi_{\cR_0}(S(T),I(T))
\end{equation*}
and $(S,I)$ solves the controlled system~\eqref{SIR} associated to the control function $u(\cdot)$.

\begin{proof}[Proof of Theorem \ref{th1}]

For better readability, the proof of Theorem \ref{th1} is decomposed into several steps.

\paragraph{Step 1: existence of an optimal control}
We will prove the existence of an optimal control for the equivalent problem~\eqref{OCP2}.
Let $(u_n)_{n\in \NN}$ be  a maximizing sequence for Problem~\eqref{OCP2}.
Since $(u_n)_{n\in \NN}$ is uniformly bounded, we may extract a subsequence still denoted $(u_n)_{n\in \NN}$ with a slight abuse of notation, converging towards $u^*$ for the weak-star topology of $L^\infty(0,T)$. It is moreover standard that $ \mathcal{U}_{\alpha,T}$ is closed for this topology and therefore, $u^*$ belongs to $ \mathcal{U}_{\alpha,T}$. For $n\in\NN$, let us denote $(S_n,I_n)$ the solution to the SIR model \eqref{SIR} associated to $u=u_n$. A straightforward application of the Cauchy-Lipschitz theorem yields that $(\dot{S}_n,\dot{I}_n)_{n\in\NN}$ is uniformly bounded. By applying Ascoli's theorem, we may extract a subsequence still denoted $(S_n,I_n)_n$ that converges towards $(S^*,I^*)$ in $C^0([0,T])$. As usually, we consider an equivalent formulation of System~\eqref{SIR}, where $(S_n,I_n)$ can be seen as the unique fixed point of an integral operator. We then pass to the limit and show that $(S^*,I^*)$ solves the same equation where $u$ has been replaced by $u^*$. By combining all these facts with the continuity of $\Phi_{\mathcal{R}_0}$, we then infer that $(J_\Phi(u_n))_{n\in\NN}$ converges up to a subsequence to $J_\Phi(u^*)$, which gives the existence. 

\paragraph{Step 2: optimality conditions and bang-bang property}
We will again establish these properties for the equivalent problem~\eqref{OCP2}.
In what follows, for the sake of simplicity, we will consider and denote by $u$ a solution to Problem~\eqref{OCP2} and by $(S,I)$, the associated pair solving System~\eqref{SIR}. Observe first that integrating \eqref{eq34} in Lemma~\ref{le1}, one has 
\begin{equation*}
J_\Phi(u) = \Phi_{\cR_0}(S(T),I(T)) = \Phi_{\cR_0}(S_0,I_0)
+\nu \int_0^T (u(t)-1)I(t)\ dt.
\end{equation*}

It is standard to write the first order optimality conditions for such kind of optimal control problem. To this aim, we use the so-called Pontryagin maximum principle (see e.g. \cite{LeeMarkus}) and introduce the Hamiltonian $\mathcal H$ defined on $\RR^4$ by 
\begin{eqnarray*}
\cH(S,I,p_S,p_I) &=& \left[
-p_S\beta  uS +  p_I (\beta uS-\nu) - \nu (u-1)
\right] I\nonumber \\
&=& \nu (1 - p_I)I + (\beta(p_I-p_S)S-\nu)Iu.
\end{eqnarray*}
There exists an absolutely continuous mapping $(p_S,p_I):[0,T]\to \RR^2$ called adjoint vector such that the extremal $((S,I),(p_S,p_I),u)$ satisfies a.e. in $[0,T]$:\\
\noindent {\it Adjoint equations and transversality conditions:}
\begin{subequations}
\begin{gather}
\label{eq7a}
\dot p_S = \beta uI(p_S-p_I),\qquad
\dot p_I = \beta uSp_S - (\beta uS-\nu )p_I + \nu (u-1), \\
\label{eq7b}
p_S(T) = 0,\qquad p_I(T) = 0.
\end{gather}
\end{subequations}
 \noindent {\it Maximization condition:} for a.e. $t\in [0,T]$, $u(t)$ solves the problem
$$
\max_{\alpha\leq v\leq 1}(\beta(p_I(t)-p_S(t))S(t)-\nu)I(t)v
$$
and therefore, by using that $I$ is nonnegative on $[0,T]$, one has
\begin{equation}\label{eq1847}
w  > -\frac{1}{\cR_0} \text{ on } \{u=\alpha\},\
w   < -\frac{1}{\cR_0} \text{ on } \{u=1\},\ 
w =- \frac{1}{\cR_0} \text{ on } \{\alpha <u<1\},
\end{equation}
where $w$ denotes the Lipschitz-continuous switching function given by $w=S(p_S-p_I)$.

By using \eqref{SIR}, one computes
\begin{eqnarray*}
 \dot w & =& -\beta uSI(p_S-p_I)+S(-\beta uI(p_I-p_S)+\beta uS(p_I-p_S)-\nu p_I - \nu(u-1))\nonumber \\
& =& -S\left(\beta u w + \nu  (p_I+u-1)\right).
\end{eqnarray*}

We will now prove that the optimal control can be written as $u_{T_0}$ defined in \eqref{eq55}, for some $T_0\in [0,T)$.
From \eqref{eq7b}, one has $p_S(T)=p_I(T)=0$, and thus $w(T)=0 > -\frac{1}{\cR_0}$. According to \eqref{eq1847}, this implies $u(\cdot)=\alpha$ on a certain {\em maximal} interval $[T_0,T]$, for some $T_0\in (0,T)$, by continuity of $w$.
By inserting the relation $w\geq -\frac{1}{\cR_0}$ holding on $(T_0,T)$ in the equation \eqref{eq7a} satisfied by $p_I$, we deduce that $\dot{p}_I \geq \nu (p_I -1)$ on $(T_0,T)$.
Since $ p_I(T) = 0$, the Gronwall lemma yields
\begin{equation}
\label{eq23}
p_I(t) \leq 1- e^{\nu(t-T)}, \quad t\in [T_0,T].
\end{equation}
Then, either $T_0=0$, in which case, $u = u_0= \alpha \mathds{1}_{[0,T]}+\mathds{1}_{[T,+\infty)}$(see \eqref{eq55}) or $T_0>0$.
Let us now address this latter case.
Since the interval $(T_0,T)$ is maximal by assumption and $w$ is continuous, one has necessarily $w(T_0)=-1/\cR_0$.
On the other hand, $S$ and $p_I$ are continuous, with $p_I(T_0)<1$ according to \eqref{eq23}.
Consequently, for any $\varepsilon$ in the non-empty open interval $(0, \nu S(T_0)(1-p_I(T_0))$, there exists a neighborhood $\mathcal V_{T_0}$ of $T_0$ on which
\begin{eqnarray*}
\dot w &=& -\beta uS\left(
w+\frac{1}{\cR_0}
\right) + \nu S(1- p_I)\nonumber \\
& \in & [\nu S(T_0)(1- p_I(T_0)) - \varepsilon, \nu S(T_0)(1-p_I(T_0)) + \varepsilon]
\quad \text{a.e.}
\end{eqnarray*}
Since $\varepsilon\in (0,\nu S(T_0)(1- p_I(T_0))$, this implies that $w$ is strictly increasing in $\mathcal V_{T_0}$.
Therefore, there exists a {\em maximal} open interval $(T_1,T_0)$ with  $T_1\in[0,T_0)$ on which $w<-\frac{1}{\cR_0}$, and therefore on which $u= 1$.
As a consequence, the left-derivative of $w$ at $T_0$ exists and reads
\begin{equation}
\label{eq22}
\dot w(T_0^-) = -\beta S(T_0) \left(
w(T_0)+\frac{1}{\cR_0} p_I(T_0)
\right) >0.
\end{equation}
We will in fact show that $T_1=0$. In other words, the control $u$ can be written as \eqref{eq55}.

To this aim, let us assume by contradiction that there exists $T_1\in (0,T_0)$ such that $w(T_1) = -\frac{1}{\cR_0} = w(T_0)$ and $w< -\frac{1}{\cR_0}$ on $(T_1,T_0)$.
Observing that $w$ is differentiable on $(T_1,T_0)$ and using Rolle's theorem yields the existence of $\tau\in (T_1,T_0)$ such that $\dot w(\tau)=0$.

Note that, according to \eqref{eq7a}, one has $\dot{w} = - S\dot p_I$, a.e.\ in $(0,T)$.
Since $S(\tau)>0$, one has also $\dot p_I (\tau)=0$.
Using the fact that $u=1$ on $(T_1,T_0)$, this means that the point $(p_I(\tau),w(\tau))$ is a {\em steady-state} of the system
\begin{equation*}
\dot w = -\beta S
w - \nu S p_I,\qquad
\dot p_I = \beta w +\nu p_I .
\end{equation*}
According to the Cauchy-Lipschitz theorem, we infer that $(p_I,w)$ is constant on $[\tau,T_0]$ and therefore, $\dot p_I(T_0^-)= \dot w(T_0^-)=0$, which is in contradiction with \eqref{eq22}.
As a conclusion, $T_1=0$, and $u$ can be written as \eqref{eq55}.

\paragraph{Step 3: monotonicity of $S_{\infty,\alpha,T}^*$}
Let $T>0$ and $0\leq \alpha \leq \tilde{\alpha} < 1$. It is straightforward that $\mathcal{U}_{\tilde{\alpha},T} \subset \mathcal{U}_{\alpha,T}$, and then $S_{\infty,\alpha,T}^* \geqslant S_{\infty,\tilde\alpha,T}^*$. It follows that the map $[0,1)\ni \alpha \mapsto \max_{u\in \mathcal{U}_{\alpha,T}} S_{\infty}(u)$ is nonincreasing.

Let us show that the map $T \mapsto \max_{u\in \mathcal{U}_{\alpha,T}} S_{\infty}(u)$ is nondecreasing.
Let $0<T\leq \tilde{T}$, $0\leq \alpha <1$, and denote by $u^*\in \mathcal{U}_{\alpha,T}$ the control realizing the maximum $S_{\infty,\alpha,T}^*$.
Since $u^*=1$ in $(T,\tilde{T})$, one has $u^*\in \mathcal{U}_{\alpha,\tilde T}$. 
Thus
\[
  S_{\infty,\alpha,\tilde T}^* \geqslant S_{\infty}(u^*)=S_{\infty,\alpha,T}^*.
\]

\paragraph{Step 4: uniqueness of the optimal control}
The demonstration of the uniqueness of the optimal control is achieved in the proof of Theorem~\ref{th:T0} below, by demonstrating the uniqueness of the optimal switching point $T_0$.
Except this point, the demonstration of Theorem~\ref{th1} is now complete.
\end{proof}

\subsection{Proof of Theorem~\ref{th:T0}}
Let us decompose this proof into several steps. 
We assume that $\alpha>0$, the case $\alpha=0$ will be considered in the last step. In the whole proof, we will deal with control functions $u_{T_0}$ as defined in formula \eqref{eq55}.

\paragraph{Step 1: necessary first order optimality conditions on $T_0$}
Let $u=u_{T_0}$ be an optimal control for problem~\eqref{OCP}, with $T_0$ be the associated optimal switching time.
Let us introduce the criterion $j_{\Phi}$ given by
\[
j_{\Phi}(T_0):=\Phi_{\mathcal{R}_0}(S^{T_0}(T),I^{T_0}(T))=I^{T_0}(T)+S^{T_0}(T) - \frac{\nu}{\beta} \ln S^{T_0}(T),
\]
where $(S^{T_0},I^{T_0})$ is the solution corresponding to the control $u_{T_0}$, as previously defined.
For the sake of simplicity, we omit these subscripts in the sequel.
Again, it is equivalent to maximize $j$ and to minimize $j_{\Phi}$.
By using Lemma~\ref{co1}, one has
\begin{equation}\label{m1440}
\begin{array}{ll}
I(t)+S(t)- \frac{\nu}{\beta} \ln S(t)=c_0 & \text{in }[0,T_0],\\
I(t)+S(t)-\frac{\nu}{\alpha\beta} \ln S(t)=I(T_0)+S(T_0)-\frac{\nu}{\alpha\beta} \ln S(T_0) & \text{in }[T_0,T],
\end{array}
\end{equation}
where $c_0=I_0+S_0- \frac{\nu}{\beta} \ln S_0$.
According to \eqref{m1440}, we infer that $S$ solves the system
\begin{subequations}\label{meq11699}
\begin{align}
& \dot S=-\beta S(c_0-S+ \frac{\nu}{\beta} \ln S), \qquad \text{in }(0,T_0), \label{S1a}\\
& \dot S=-\alpha\beta S\left(c_0+\frac{\nu}{\beta}\left(1-\frac{1}{\alpha} \right)\ln S(T_0)-S+\frac{\nu}{\alpha\beta} \ln S\right), \qquad \text{in }(T_0,T),  \label{S1b}
\end{align}
\end{subequations}
with the initial data $S(0)=S_0$.
Using \eqref{m1440}, one gets 
\begin{eqnarray*}
j_{\Phi}(T_0) &=& I(T)+S(T)- \frac{\nu}{\beta} \ln S(T)\\
&=& I(T_0)+S(T_0)-\frac{\nu}{\alpha\beta} \ln S(T_0)+\frac{\nu}{\beta}\left(\frac{1}{\alpha}-1\right) \ln S(T)\\
       &=& c_0+ \frac{\nu}{\beta} \ln S(T_0)-\frac{\nu}{\alpha\beta}\ln S(T_0)+\frac{\nu}{\beta}\left(\frac{1}{\alpha}-1\right) \ln S(T),
\end{eqnarray*}
so that the cost function reads
\[
  j_{\Phi}(T_0) = c_0+\frac{\nu}{\beta}\left(\frac{1}{\alpha}-1\right) \ln \left(\frac{S(T)}{S(T_0)}\right).
\]

The next lemma allows to compute the derivative of $j_\Phi$ with respect to $T_0$. For the sake of clarity, its proof is postponed to the end of this section.
\begin{lemma}\label{lem:HatS2}
For all $t\in [T_0,T]$, the derivative $\widehat{S}(t)$ and $\widehat{S(T_0)}$ of the function $S(\cdot)$ and $S(T_0)$  with respect to $T_0$\footnote{To avoid any misunderstanding about the differentiability of $S(\cdot)$ with respect to $T_0$, let us make the use of $\widehat{S}$ precise. This function stands for the derivative of the function $[0,T]\ni T_0\mapsto S(\cdot,T_0)\in C^0 ([T_0,T])$, where $S(\cdot,T_0)$ is defined as the unique solution to \eqref{S1b} on $[T_0,T]$, where $S(T_0)$ is defined as the value at $T_0$ of the unique solution to \eqref{S1a}.
Defined in this way, the differentiability of this mapping is standard.}, in other words $\widehat{S}(t) = \frac{\partial S(t)}{\partial T_0}$ and $\widehat{S(T_0)} = \frac{\partial [S(T_0)]}{\partial T_0}$,  are given by
\begin{equation*}
\widehat{S}(t)=(\alpha-1)\beta S(t)I(t)\left(1+\nu I(T_0)\int_{T_0}^{t}\frac{ds}{I(s)}\right)
\end{equation*}
and
\[\widehat{S(T_0)}=-\beta S(T_0)I(T_0).\]
\end{lemma}
Thanks to this result, we may compute
\begin{eqnarray*}j_\Phi'(T_0) 
 & = &
 \frac{\nu}{\beta}\left(\frac{1}{\alpha}-1\right)\left(\frac{\widehat{S}(T)}{S(T)}-\frac{\widehat{S(T_0)}}{S(T_0)}\right)\\
 & = &
 \nu\left(\frac{1}{\alpha}-1\right)\left((\alpha-1)I(T)\left(1+\nu I(T_0)\int_{T_0}^{T}\frac{ds}{I(s)}\right)+I(T_0)\right) \\
 & = &
 \nu\left(\frac{1}{\alpha}-1\right)I(T_0)(\alpha-1)\left(\frac{I(T)}{I(T_0)}+\nu \int_{T_0}^{T}\frac{I(T)}{I(s)}ds-\frac{1}{1-\alpha}\right)  .
\end{eqnarray*}
By noting that
  \begin{align*}
  \int_{T_0}^T \frac{S(t)}{I(t)}\,dt =\ &
  \frac{1}{\alpha\beta} \int_{T_0}^T \frac{\alpha \beta S(t) - \nu + \nu}{I(t)}\,dt
  = \frac{1}{\alpha\beta} \int_{T_0}^T \frac{\dot I(t)}{I(t)^2}\,dt + \frac{\nu}{\alpha\beta} \int_{T_0}^T \frac{1}{I(t)}\,dt  \\
  =\ & \frac{1}{\alpha\beta}\left(\frac{1}{I(T_0)} - \frac{1}{I(T)} + \nu \int_{T_0}^T \frac{1}{I(t)}\,dt\right),
  \end{align*}
we have for the function $\psi$ defined in \eqref{eq:psi}:
\begin{equation}\label{eq:psi new}
\psi(T_0)
 =\Big(\frac{1}{\alpha} -1\Big)\left(\frac{I(T)}{I(T_0)}+\nu \int_{T_0}^{T}\frac{I(T)ds}{I(s)}-\frac{1}{1-\alpha}\right).
\end{equation}
We deduce that $j_{\Phi}'(T_0)=0$ is equivalent to 
\begin{equation}\label{eq:cond psi}
\psi(T_0)=0.
\end{equation}

\paragraph{Step 2: Zeros of $j_{\Phi}'$ and uniqueness of the optimal switching time}
 According to \eqref{SIR}, one has for any  $t\in (T_0,T)$,
  $I(t) = I(T_0) \exp\left(\int_{T_0}^t (\alpha\beta S(s)-\nu)\,ds\right)$.
  Then, 
 using the expression of $\psi$ given in \eqref{eq:psi new},  it follows that
  \begin{align*}
 \psi(T_0) =\ & \Big(\frac{1}{\alpha} -1\Big) \exp\left(\int_{T_0}^T (\alpha\beta S(s)-\nu)\,ds\right)  \\
    & + \Big(\frac{1}{\alpha} -1\Big) \nu \int_{T_0}^T \exp\left(\int^{T}_t (\alpha\beta S(s)-\nu)\,ds\right)\,dt   - \frac{1}{\alpha}.
  \end{align*}
  Introducing $\varphi:[0,T]\to \RR$ defined by $\varphi(s)=\exp\left(\int_{s}^T (\alpha\beta S(s)-\nu)\,ds\right)$, the last expression writes simply
  $$
 \psi(T_0) = \Big(\frac{1}{\alpha} -1\Big) \varphi(T_0)
  + \Big(\frac{1}{\alpha} -1\Big) \nu \int_{T_0}^T \varphi(t)\,dt   - \frac{1}{\alpha}.
  $$
  Differentiating this identity with respect to $T_0$ yields
\begin{eqnarray*}
   \psi'(T_0) &=&
   \Big(\frac{1}{\alpha} -1\Big)\left(-\alpha\beta S(T_0) + \nu + \int_{T_0}^T \alpha \beta \widehat{S}(s)\,ds \right) \varphi(T_0)\\
   & & - \nu  \Big(\frac{1}{\alpha} -1\Big) \varphi(T_0) + \nu \Big(\frac{1}{\alpha} -1\Big) \int_{T_0}^T \left(\int_t^T \alpha\beta \widehat{S}(s)\,ds\right) \varphi(t)\, dt  \\
               &= & \Big(\frac{1}{\alpha} -1\Big) \left(-\alpha\beta S(T_0) + \int_{T_0}^T \alpha \beta \widehat{S}(s)\,ds\right)\varphi(T_0)  \\
    && + \nu \Big(\frac{1}{\alpha} -1\Big)\int_{T_0}^T \left(\int_t^T \alpha\beta \widehat{S}(s)\,ds\right) \varphi(t)\,dt.
  \end{eqnarray*}

As a consequence of Lemma \ref{lem:HatS2}, both terms in the previous formula are negative, and $\psi'(T_0) < 0$.
The function $\psi$ is thus decreasing on $(0,T_0)$.
Moreover, $\psi(T) = -1<0$.
  Therefore, if $\psi(0)<0$, then \eqref{eq:cond psi}, or equivalently $j_{\Phi}'(T_0)=0$, has no solution, and thus $T_0^*=0$.
Conversely,  if  $\psi(0)\geq0$, then \eqref{eq:cond psi} admits a unique solution $T_0^*$ which is the unique critical point of $j_{\Phi}$, and of $j$.
In particular, in the case    $\psi(0)=0$, one has $T_0^*=0$.
 
 We also deduce that the function $j_{\Phi}$ is nonincreasing on $(0,T_0^*)$ and increasing on $(T_0^*,T)$, and so is $j$.

\paragraph{Step 3: Upper bound on $T_0$}
 Let us now prove that $S^T(T_0^*)\geqslant S_\herd$, where, in accordance with the notation defined after \eqref{eq55}, $(S^T,I^T)$ is the solution to System \eqref{SIR} for $u(\cdot)=1$.
Assume by contradiction that $S^T(T_0^*)< S_\herd$, i.e. $T_0^*>(S^T)^{-1}(S_\herd)$.
Since $I$ is non-increasing on $((S^T)^{-1}(S_\herd),\infty)$, we have 
$$I^T((S^T)^{-1}(S_\herd))>I^T(T_0^*).$$
Due to the fact (obtained from Lemma \ref{le1}) that
$$
\Phi_{\cR_0}(S^T((S^T)^{-1}(S_\herd)),I^T((S^T)^{-1}(S_\herd))) = \Phi_{\cR_0}(S^T(T_0^*),I^T(T_0^*)),
$$
we deduce that
\begin{equation}\label{eq: decroi}
  S_\herd-\frac{1}{\mathcal{R}_0}\ln(S_\herd)>S^T(T_0^*) -\frac{1}{\mathcal{R}_0}\ln(S^T(T_0^*)).
\end{equation}
 Since $S^T(T_0^*)\in(0,S_\herd)$ and $x\mapsto x-\frac{1}{\mathcal{R}_0}\ln(x)$ is decreasing on $(0,S_\herd)$, the inequality \eqref{eq: decroi} implies that $S^T(T_0^*)> S_\herd$, which leads to a contradiction.
 We have thus demonstrated that $S^T(T_0^*)\geqslant S_\herd$.

\paragraph{Step 4: The case $\alpha=0$}
Let us finally deal with the case ``$\alpha=0$''.
Using the fact that $S$ is constant on $(T_0,T)$, we deduce that 
$p_S(t)=0$ and $p_I(t)=1-e^{\nu (T_0-T)}$  for all $t\in(T_0,T)$. A commutation occurs at $T_0$  if, and only if,
\begin{equation}\label{eq:alpha 0}
S_\herd=w(T_0)=S(T_0)(1-e^{\nu (T_0-T)}).
\end{equation}
  The function $S$ is nonincreasing, thus there exists $T_0>0$ satisfying this relation only if $S_0 > \frac{S_\herd}{1-e^{-\nu T}}$ which is equivalent to
  $T> \frac{1}{\nu} \ln \frac{S_0}{S_0-S_\herd}$.
  If this is the case, then, since $t\mapsto S(t)$ is nonincreasing and $T_0\mapsto \frac{S_\herd}{1-e^{\nu(T_0-T)}}$ is increasing, there exists a unique $T_0$ satisfying the relation \eqref{eq:alpha 0}. We also remark that \eqref{eq:alpha 0} is equivalent to \eqref{eq:cond psi}.

%


To achieve the proof of Theorem~\ref{th:T0}, it now remains to prove Lemma~\ref{lem:HatS2}.

\begin{proof}[Proof of Lemma~\ref{lem:HatS2}]
Using the notation $S^{T_0}$ previously defined, one has (see \eqref{S1b}) on $(T_0,T)$
\begin{equation*}
\dot S^{T_0}=-\alpha\beta  S^{T_0}\left(c_0+\frac{\nu}{\beta }\left(1-\frac{1}{\alpha} \right)\ln (S^{T_0}(T_0))-S^{T_0}+\frac{\nu}{\alpha\beta } \ln S^{T_0}\right),
\end{equation*}
and, at $T_0$, $S^{T_0}(T_0)$ is defined thanks to \eqref{S1a} by 
\begin{equation}\label{m008bis}
\int_{S_0}^{S^{T_0}(T_0)}\frac{dv}{\beta  v(c_0-v+\frac{\nu}{\beta }\ln v)}=-T_0.
\end{equation}
By differentiating \eqref{m008bis} with respect to $T_0$, one infers
\[
\widehat{S^{T_0}(T_0)}=-\beta S^{T_0}(T_0)\left(c_0-S^{T_0}(T_0)+ \frac{\nu}{\beta } \ln (S^{T_0}(T_0))\right)=-\beta S^{T_0}(T_0)I^{T_0}(T_0),
\]
which is the second identity in Lemma~\ref{lem:HatS2}.

Furthermore,  using \eqref{S1b}, one has
\begin{equation*}
\int_{S^{T_0}(T_0)}^{S^{T_0}(t)}\frac{dv}{v(c_0+\frac{\nu}{\beta }\left(1-\frac{1}{\alpha} \right)\ln (S^{T_0}(T_0))-v+\frac{\nu}{\alpha\beta }\ln v)}=\alpha\beta (T_0-t),\quad t\in [T_0,T].
\end{equation*}
Differentiating this relation with respect to $T_0$ yields for all $t\in(T_0,T]$
\begin{multline*}
\alpha\beta  =
\frac{\widehat{S^{T_0}}(t)}{S^{T_0}(t)(c_0+\frac{\nu}{\beta }\left(1-\frac{1}{\alpha} \right)\ln S^{T_0}(T_0)-S^{T_0}(t)+\frac{\nu}{\alpha\beta }\ln S^{T_0}(t))}\\
-\frac{\widehat{S^{T_0}(T_0)}}{S^{T_0}(T_0)(c_0+\frac{\nu}{\beta }\ln S^{T_0}(T_0)-S^{T_0}(T_0))}\\
-\frac{\nu}{\beta } \left(1-\frac{1}{\alpha}\right)\frac{\widehat{S^{T_0}(T_0)}}{S^{T_0}(T_0)}\int_{S^{T_0}(T_0)}^{S^{T_0}(t)}\frac{dv}{v(c_0+\frac{\nu}{\beta }\left(1-\frac{1}{\alpha} \right)\ln S^{T_0}(T_0)-v+\frac{\nu}{\alpha\beta }\ln v)^2}.
\end{multline*}

Let us simplify this latter identity.
Observe first that, because of \eqref{m1440}, one has for all $t\in(T_0,T]$
\begin{align*}
  c_0+\frac{\nu}{\beta } \left(1-\frac{1}{\alpha}\right)\ln S^{T_0}(T_0)
  & = I^{T_0}(t)+S^{T_0}(t)-\frac{\nu}{\alpha\beta } \ln S^{T_0}(t).
\end{align*}
By using at the same time the change of variable $v=S(t)$ and the identity
$
c_0+\frac{\nu}{\beta }\left(1-\frac{1}{\alpha} \right)\ln S(T_0)-S(t)+\frac{\nu}{\alpha\beta }\ln S(t) = I(t),
$ holding true for any $t\in (T_0,T]$, we infer that
\begin{multline*}
\int_{S^{T_0}(T_0)}^{S^{T_0}(t)} \frac{dv}{\beta v(c_0+\frac{\nu}{\beta }\left(1-\frac{1}{\alpha} \right)\ln S^{T_0}(T_0)-v+\frac{\nu}{\alpha\beta }\ln v)^2}\\
= \int_{T_0}^t \frac{1}{\beta S^{T_0}(s)(I^{T_0}(s))^2} \dot{S}^{T_0}(s) \,ds   
= \int_{T_0}^t \frac{-\alpha S^{T_0}(s)I^{T_0}(s)}{S^{T_0}(s)(I^{T_0}(s))^2}\,ds
= -\alpha \int_{T_0}^t \frac{ds}{I^{T_0}(s)}. 
\end{multline*}
Combining all these facts leads to, for all $t\in(T_0,T]$,
  \begin{align*}
    \alpha\beta = \frac{\widehat{S^{T_0}}(t)}{S^{T_0}(t) I^{T_0}(t)} + \beta - \nu (1-\alpha)\beta I^{T_0}(T_0) \int_{T_0}^t \frac{ds}{I^{T_0}(s)}.
  \end{align*}
Therefore, we arrive at, for all $t\in(T_0,T]$,
\begin{equation*}
\widehat{S^{T_0}}(t)=(\alpha-1)\beta S^{T_0}(t)I^{T_0}(t)\left(1+\nu I^{T_0}(T_0)\int_{T_0}^{t}\frac{ds}{I^{T_0}(s)}\right),
\end{equation*}
which is the first identity of the statement.
This achieves the proof of Lemma~\ref{lem:HatS2}, and consequently of Theorem~\ref{th:T0}.
\end{proof}

\subsection{Proof of Theorem \ref{theo:min time}}
Let $T^*$ be the minimal time associated to the Minimal time problem. Let $u^*_{\alpha,T^*}$ be the unique solution of Problem~\eqref{OCP} associated to $T=T^*$.
Assume by contradiction that $u^*_{\alpha,T^*}$ does not solve the minimal time problem, i.e.\ $S_{\infty}(u^*_{\alpha,T^*}) < S_\herd - \eps$. 
 Then for each $u\in\mathcal{U}_{\alpha,T^*}$ one has
 \[
   S_{\infty}(u)\leqslant S_{\infty}(u^*_{\alpha,T^*}) < S_\herd - \eps,
 \]
which is in contradiction with the fact that $T^*$ solves the minimal time problem.

Conversely, let $T>0$ and $u^*$ realizing the maximum of \eqref{OCP}, i.e.\ $S_\infty(u^*) = S_{\infty,\alpha,T}^*$. 
Let $T^*$ be the minimal time of intervention such that $S_{\infty}(u) \geqslant S_{\infty,\alpha,T}^*$ for some $u\in\mathcal{U}_{\alpha,T^*}$. Since $S_\infty(u^*) = S_{\infty,\alpha,T}^*$, we necessarily have $T^*\leqslant T$. 
By definition of $T^*$, there exists $u\in\mathcal{U}_{\alpha,T^*}$ such that $S_{\infty}(u) \geqslant S_{\infty,\alpha,T}^*$.
Consider $v :=u\mathds{1}_{(0,T^*)} + \mathds{1}_{(T^*,T)}$. One has $v\in \mathcal{U}_{\alpha,T}$ and 
\[
  S_\infty(v) =S_\infty(u) \geqslant S_{\infty,\alpha,T}^*.
\]
Hence, by definition of $S_{\infty,\alpha,T}^*$, one has $v=u^*$.
But $v$ does not have the form of the minimal solution in Theorem \ref{th1}, unless $T^*=T$.

\section{Numerical illustrations}
\label{se24}

This section is devoted to numerical illustrations around the optimal control problem \eqref{OCP}. 
The codes are available on:
\begin{center} \smallskip
\url{https://github.com/michelduprez/optimal-immunity-control.git}\smallskip
\end{center}

We use the parameter values given in Table \ref{tab:value}, coming from \cite{Salje:2020aa} and corresponding to the lockdown conditions in force in France between March 17th and May 11th 2020.
We suppose that, on the total population of $6.7e7$ persons in France, there are no removed individuals and $1000$ infected individuals at the initial time, i.e.\ $R_0=0$ and $I_0=1e3/6.7e7$. 
\begin{table}[ht]
\centering
\begin{tabular}{|c|l|c|c|}
\hline
\textit{Parameter}&\textit{Name}&\textit{Value}\\\hline
 $\beta$&Probability of transmission&0.29  \\\hline
 $\nu$&Rate at which the infected are removed&0.1 \\ \hline
 $\alpha_{\mbox{\tiny lock}}$&Lockdown in France (March-May 2020)&0.231 \\\hline
 $N$&Total population in France&6.7e7\\\hline
  $S_0$&Initial proportion of susceptible&$1-I_0$  \\ \hline
   $I_0$&Initial proportion of  infected&$1e3/N$  \\ \hline
    $R_0$&Initial proportion of removed&0  \\ \hline
\end{tabular}
\caption{Value of the parameters used for system \eqref{eq1a}-\eqref{eq1b} (see \cite{Salje:2020aa})}
\label{tab:value}
\end{table}

Solutions to ODEs are computed with the help of a Runge-Kutta fourth-order method
and the research of the optimal controls is conducted thanks to the Algorithms described below.

Our approach rests upon the use of gradient like algorithms needing the computation of the differential of $S_\infty$ in an admissible direction\footnote{More precisely, we call ``admissible direction'' any element of the tangent cone $\mathcal{T}_{u,\cU_{\alpha, T}}$ to the set $\cU_{\alpha, T}$ at $u$. The cone $\mathcal{T}_{u,\cU_{\alpha, T}}$ is the set of functions $h\in L^\infty(0,T)$ such that, for any sequence of positive real numbers $\varepsilon_n$ decreasing to $0$, there exists a sequence of functions $h_n\in L^\infty(0,T)$ converging to $h$ as $n\rightarrow +\infty$, and $u+\varepsilon_nh_n\in\cU_{\alpha, T}$ for every $n\in\NN$.
} $h$. 
According to the proof of Theorem~\ref{th1} (see below Section \ref{se31}), this differential reads
\[
DS_\infty(u)\cdot h=\int_0^T(\nu-\beta S(p_I-p_S))Ih \,dt,
\]
where $(p_S,p_I)$ denotes the adjoint state, solving the backward adjoint system \eqref{eq7a}-\eqref{eq7b}.
Thanks to this expression of $DS_\infty(u)\cdot h$, we deduce a simple projected gradient algorithm to solve numerically the optimal control problem \eqref{OCP2}, then \eqref{OCP}.
The algorithm is described in Algorithm \ref{algo:full}.
The projection operator $\mathbb{P}_{\mathcal{U}_{\alpha,T}}$ is given by
\[
\mathbb{P}_{\mathcal{U}_{\alpha,T}}u(t)=\min\{\max\{u(t),\alpha\},1\}, \quad \text{for a.e. }t\in [0,T].
\]

\begin{algorithm}
\caption{Solving Problem \eqref{OCP} by projected gradient descent}
\label{algo:full}
\begin{algorithmic}[1]
\REQUIRE{$u_0\in\mathcal{U}_{\alpha,T}$, $\varepsilon>0$}
\WHILE{$|S_\infty (u_{k+1})-S_\infty (u_k)|>\varepsilon$}
 \STATE Compute $(S_k,I_k)$ solution to the primal system \eqref{SIR}
  \STATE Compute $(p_{S,k},p_{I,k})$ solution to the dual system \eqref{eq7a}-\eqref{eq7b}
 \STATE Compute 
 $u_{k+1}:=\mathbb{P}_{\mathcal{U}_{\alpha,T}}(u_k-\rho_k DS_\infty(u_k))$ 
where $\rho_k$ is the step, chosen variable such that $S_\infty (u_{k+1})-S_\infty (u_k)<0$.
\ENDWHILE\RETURN $u_{k+1}$
\end{algorithmic}
\end{algorithm}

We recall that it is equivalent to maximize $j$ and to minimize $j_\Phi$.
According to Theorem \ref{th1}, solutions $u^*$ and $T_0^*$ to \eqref{OCP} and \ref{eq:opt T0}, respectively, satisfy $u^*=u_{T^*_0}$, moreover  $j_{\Phi}$ is decreasing on $(0,T_0^*)$ and increasing on $(T_0^*,T)$.  
Taking advantage of this property, we also test a second algorithm based on the bisection method.
This method is described in Algorithm~\ref{algo:finite dim}.

\begin{algorithm}
\caption{Solving Problem \eqref{eq:opt T0} (then \eqref{OCP}) by bisection method}
\label{algo:finite dim}
\begin{algorithmic}[1]
\REQUIRE{$k\in\mathbb{N}^*$}

\STATE\textbf{Initialization:} $T_{0,\mbox{\tiny min},0}=0$ and $T_{0,\mbox{\tiny max},0}=T$

\FOR{$i=1,\dots, k$}
\STATE $T_{0,\mbox{\tiny left}}= T_{0,\mbox{\tiny min},i}+(T_{0,\mbox{\tiny max},i}-T_{0,\mbox{\tiny min},i})/3$
\STATE  $T_{0,\mbox{\tiny right}}=T_{0,\mbox{\tiny min},i}+2(T_{0,\mbox{\tiny max},i}-T_{0,\mbox{\tiny min},i})/3$
 \STATE Compute $(S_{\mbox{\tiny left}},I_{\mbox{\tiny left}})$, $(S_{\mbox{\tiny right}},I_{\mbox{\tiny right}})$ solutions to \eqref{eq11} for $u_{T_{0,\mbox{\tiny left}}}$, $u_{T_{0,\mbox{\tiny right}}}$
  
  \IF{$j_\Phi(T_{0,\mbox{\tiny left}})\geq j_\Phi(T_{0,\mbox{\tiny left}})$}
  \STATE $T_{0,\mbox{\tiny min,i+1}}= T_{0,\mbox{\tiny left}}$
  and $T_{0,\mbox{\tiny max,i+1}}=T_{0,\mbox{\tiny max,i}}$
  \ELSE\STATE $T_{0,\mbox{\tiny min,i+1}}=T_{0,\mbox{\tiny min,i}}$
and  $T_{0,\mbox{\tiny max}}= T_{0,\mbox{\tiny right}}$
  \ENDIF
\ENDFOR
\RETURN $T_{0,k+1}:=(T_{0,\mbox{\tiny max},k+1}+T_{0,\mbox{\tiny min},k+1})/2$, $u_{k+1}:=u_{T_{0,k+1}}$
\end{algorithmic}
\end{algorithm}

Let us compare the numerical solutions obtained with these approaches for each problem \eqref{OCP} and  \eqref{eq:opt T0}.
Optimal trajectories denoted respectively $(S^{u^*},I^{u^*},R^{u^*})$ and $(S^{T_0^*},I^{T_0^*},R^{T_0^*})$ associated to the optimal control problems \eqref{OCP} and \eqref{eq:opt T0} are provided on Figures \ref{fig:u opt alpha=0} and \ref{fig:compar sol} for the parameter choices $\alpha\in\{0,\alpha_{\mbox{\tiny lock}}\}$ and $T=100$. To capture the  full behavior of the trajectories, the computation window is $[0,200]$.
One recovers all the theoretical results: the solutions to \eqref{OCP} and \eqref{eq:opt T0} are as expected and do coincide.

By using Lemma~\ref{co1}, it is easy to determine numerically the optimal value $S_{\infty,\alpha,T}^*$ by solving the equation
\begin{equation*}
S^{u^*}(T) + I^{u^*}(T) - \frac{\nu}{\beta} \ln S^{u^*}(T)=
S_{\infty,\alpha,T}^*- \frac{\nu}{\beta} \ln S_{\infty,\alpha,T}^*.
\end{equation*}
This allows to investigate numerically on Fig.~\ref{fig:Sinf} and \ref{fig:T0} the dependency of $S_{\infty,\alpha,T}^*$ and $T_0^*$ with respect to the parameters $T$ and $\alpha$.
On Fig. \ref{fig:Sinf}, for $T=400$, we observe numerically that the lower bound $\overline{\alpha}$ 
given in Proposition \ref{prop:Sinf} and below which $S$ can get as close as we want to $S_\herd$ over an infinite horizon, is optimal ($\overline{\alpha}\approx 0.56$).
In particular, for lockdown conditions similar to the ones in effect in France between March and May 2020 ($\alpha\approx 0.231$), it appears that it is possible to come as close as we want to the optimal bound $S_\herd$ of inequality \eqref{eq:bound S}.
Interestingly, we observe on Fig. \ref{fig:T0}-left that when $\alpha$ is not small enough we have $T_0=0$ for $T$ large enough.

We also mention that Fig \ref{fig:Sinf}-left
gives the solution to the minimal time problem. Indeed, from Theorem \ref{theo:min time}, a control $u^*$ is optimal for \eqref{OCP} iff it is optimal for the minimal time problem. Then, given $\eps>0$ and $\alpha\leq \overline{\alpha}$, the minimal time of action such that the final value of susceptible is at a distance $\eps$ of $S_\herd$ is obtained by computing the intersection of the curves in Fig \ref{fig:Sinf}-left with the horizontal line $S_\herd-\eps$.
As expected, when $\alpha$ is too large, i.e.\ when the lockdown is insufficient, the solution stays far away from $S_\herd$ (see Fig \ref{fig:Sinf}-right).

Figure \ref{fig:values J} represents  the function $j_\Phi$ for different values of $\alpha$ and $T$. We recover the fact that $j_\Phi$ is decreasing then increasing stated in Theorem \ref{th:T0}.
The fact that $j_\Phi$ is non-convex highlights the difficulty of the considered optimal control problem.

\begin{figure}[ht]
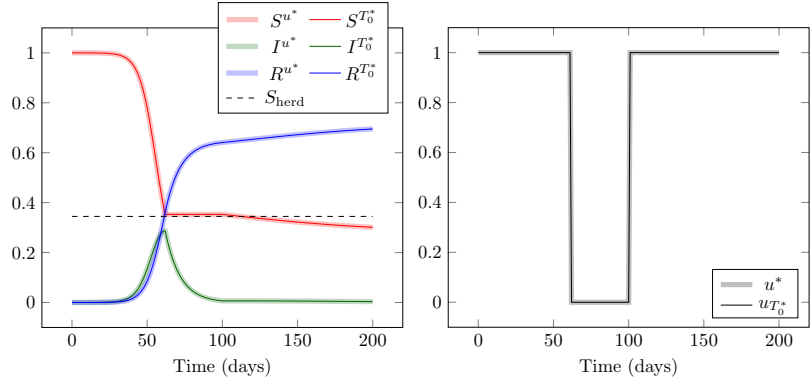
 
\begin{center}

\end{center}
\caption{Solutions $(S^{u^*},I^{u^*},R^{u^*},u^*)$ and $(S^{T_0^*},I^{T_0^*},R^{T_0^*},u_{T_0^*})$ to Problems \eqref{OCP} and \eqref{eq:opt T0} for $\alpha=0.0$ and $T=100$ on the time interval $[0,200]$. 
In this case $S_{\infty}^*=0.282$ and $T_0^*=61.9$.
\label{fig:u opt alpha=0}}
\end{figure}

\begin{figure}[ht] 
\begin{center}
%

\end{center}
\caption{
Solutions $(S^{u^*},I^{u^*},R^{u^*},u^*)$ and $(S^{T_0^*},I^{T_0^*},R^{T_0^*},u_{T_0^*})$ to Problems \eqref{OCP} and \eqref{eq:opt T0} for $\alpha=\alpha_{\mbox{\tiny lock}}$ and $T=100$ on the time interval $[0,200]$.
In this case $S_{\infty}^*=0.259$ and $T_0^*=59.2$
\label{fig:compar sol}}
\end{figure}

\begin{figure}[ht] 
\begin{center}
%

\end{center}
\caption{
Graph of $S_{\infty,\alpha,T}^*$  for Problem~\eqref{OCP} with respect to $T$  for $\alpha\in\{0,\alpha_{\mbox{\tiny lock}},0.7,0.8\}$ (left) and with respect to $\alpha$  for $T\in\{100,200,400\}$ (right) (opt.: Algo \ref{algo:finite dim}; theo.: Formula \eqref{eq:formula}).
\label{fig:Sinf}}
\end{figure}

\begin{figure}[ht] 
\begin{center}
%

\end{center}
\caption{
Graph of $T_0^*$  for Problem~\eqref{OCP} with respect to $T$  for $\alpha\in\{0.0,\alpha_{\mbox{\tiny lock}},0.7,0.8\}$ (left) and with respect to $\alpha$  for $T\in\{100,200,400\}$ (right) (opt.: Algo \ref{algo:finite dim}; theo.: Formula \eqref{eq:formula}).\label{fig:T0}}
\end{figure}

\begin{figure}[ht]
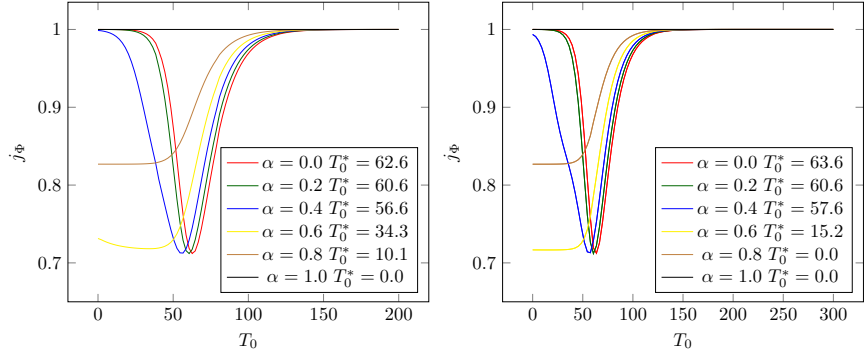
 
\begin{center}
%

\end{center}
\caption{
Graph of the optimal value $j_\Phi$ with respect to $T_0$ for $T=200$ (left) and $T=300$ (right).
}\label{fig:values J}
\end{figure}

\section*{Acknowledgments}

The authors express their sincere acknowledgment to 
working groups Maths4Covid19 and OptimCovid19 for fruitful discussions and in particular
their colleagues Luis Almeida (CNRS UMR 7598, LJLL, France), Emmanuel Franck (INRIA Grand-Est and IRMA Strasbourg, France), Sidi-Mah\-moud Kaber (Sorbonne Universit\'e, LJLL, France), Gr\'egoire Nadin (CNRS UMR 7598, LJLL, France), Beno\^\i t Perthame (Sorbonne Universit\'e, LJLL, France).

\bibliographystyle{amsalpha}
\bibliography{biblio-Opt-Lock}

\newcommand{\etalchar}[1]{$^{#1}$}
\providecommand{\bysame}{\leavevmode\hbox to3em{\hrulefill}\thinspace}
\providecommand{\MR}{\relax\ifhmode\unskip\space\fi MR }
\providecommand{\MRhref}[2]{%
  \href{http://www.ams.org/mathscinet-getitem?mr=#1}{#2}
}
\providecommand{\href}[2]{#2}
\begin{thebibliography}{JDHWM13}

\bibitem[AAL20]{Alvarez:2020aa}
Fernando~E Alvarez, David Argente, and Francesco Lippi, \emph{A simple planning
  problem for covid-19 lockdown}, Tech. report, National Bureau of Economic
  Research, 2020.

\bibitem[Aba73]{Abakuks:1973aa}
Andris Abakuks, \emph{An optimal isolation policy for an epidemic}, Journal of
  Applied Probability \textbf{10} (1973), no.~2, 247--262.

\bibitem[Aba74]{Abakuks:1974aa}
\bysame, \emph{Optimal immunisation policies for epidemics}, Advances in
  Applied Probability \textbf{6} (1974), no.~3, 494--511.

\bibitem[AER{\etalchar{+}}14]{Alkama:2014aa}
M~Alkama, M~Elhia, Z~Rachik, M~Rachik, and E~Labriji, \emph{Free terminal time
  optimal control problem of an {SIR} epidemic model with vaccination},
  International journal of science and research \textbf{3} (2014), 227.

\bibitem[AI12]{Ainseba:2012aa}
B~Ainseba and M~Iannelli, \emph{Optimal screening in structured {SIR}
  epidemics}, Mathematical Modelling of Natural Phenomena \textbf{7} (2012),
  no.~3, 12--27.

\bibitem[BBDMG19]{Bolzoni:2019aa}
Luca Bolzoni, Elena Bonacini, Rossella Della~Marca, and Maria Groppi,
  \emph{Optimal control of epidemic size and duration with limited resources},
  Mathematical biosciences \textbf{315} (2019), 108232.

\bibitem[BBSG17]{Bolzoni:2017aa}
Luca Bolzoni, Elena Bonacini, Cinzia Soresina, and Maria Groppi,
  \emph{Time-optimal control strategies in {SIR} epidemic models}, Mathematical
  biosciences \textbf{292} (2017), 86--96.

\bibitem[Beh00]{Behncke:2000aa}
Horst Behncke, \emph{Optimal control of deterministic epidemics}, Optimal
  control applications and methods \textbf{21} (2000), no.~6, 269--285.

\bibitem[DB80]{Di-Blasio:1980aa}
G~Di~Blasio, \emph{A synthesis problem for the optimal control of epidemics},
  Numerical Functional Analysis and Optimization \textbf{2} (1980), no.~5,
  347--359.

\bibitem[Gre88]{Greenhalgh:1988aa}
David Greenhalgh, \emph{Some results on optimal control applied to epidemics},
  Mathematical Biosciences \textbf{88} (1988), no.~2, 125--158.

\bibitem[GS09]{Gaff:2009aa}
Holly Gaff and Elsa Schaefer, \emph{Optimal control applied to vaccination and
  treatment strategies for various epidemiological models}, Mathematical
  Biosciences \& Engineering \textbf{6} (2009), no.~3, 469.

\bibitem[HD11]{Hansen:2011aa}
Elsa Hansen and Troy Day, \emph{Optimal control of epidemics with limited
  resources}, Journal of mathematical biology \textbf{62} (2011), no.~3,
  423--451.

\bibitem[HZ13]{Hu:2013aa}
Qingwen Hu and Xingfu Zou, \emph{Optimal vaccination strategies for an
  influenza epidemic model}, Journal of Biological Systems \textbf{21} (2013),
  no.~04, 1340006.

\bibitem[JDHWM13]{Jaberi-Douraki:2013ab}
Majid Jaberi-Douraki, Jane~M Heffernan, Jianhong Wu, and Seyed~M Moghadas,
  \emph{Optimal treatment profile during an influenza epidemic}, Differential
  Equations and Dynamical Systems \textbf{21} (2013), no.~3, 237--252.

\bibitem[JDM13]{Jaberi-Douraki:2013aa}
Majid Jaberi-Douraki and Seyed~M Moghadas, \emph{Optimality of a time-dependent
  treatment profile during an epidemic}, Journal of biological dynamics
  \textbf{7} (2013), no.~1, 133--147.

\bibitem[KGZ15]{Kolesin:2015aa}
Igor~D Kolesin, Elena Gubar, and Ekaterina Zhitkova, \emph{Optimal regime of
  double vaccination}, 2015 International Conference "Stability and Control
  Processes" in Memory of VI Zubov (SCP), IEEE, 2015, pp.~248--251.

\bibitem[KS20]{Kruse:2020aa}
Thomas Kruse and Philipp Strack, \emph{Optimal control of an epidemic through
  social distancing}, Tech. report, April 2020.

\bibitem[KZ16]{Kolesin:2016aa}
Igor~D Kolesin and Ekaterina~M Zhitkova, \emph{Optimization of immunocorrection
  of collective immunity}, Automation and Remote Control \textbf{77} (2016),
  no.~6, 1031--1040.

\bibitem[LM67]{LeeMarkus}
E.~B. Lee and L.~Markus, \emph{Foundations of optimal control theory}, Wiley
  New York, 1967 (English).

\bibitem[LT15]{Laguzet:2015aa}
Laetitia Laguzet and Gabriel Turinici, \emph{Globally optimal vaccination
  policies in the {SIR} model: smoothness of the value function and uniqueness
  of the optimal strategies}, Mathematical biosciences \textbf{263} (2015),
  180--197.

\bibitem[MRPL20]{Morris:2020aa}
Dylan~H Morris, Fernando~W Rossine, Joshua~B Plotkin, and Simon~A Levin,
  \emph{Optimal, near-optimal, and robust epidemic control}, arXiv preprint
  arXiv:2004.02209 (2020).

\bibitem[MW74]{Morton:1974aa}
Richard Morton and Kenneth~H Wickwire, \emph{On the optimal control of a
  deterministic epidemic}, Advances in Applied Probability \textbf{6} (1974),
  no.~4, 622--635.

\bibitem[PS20]{Piguillem:2020aa}
Facundo Piguillem and Liyan Shi, \emph{The optimal {COVID-19} quarantine and
  testing policies}, Tech. report, Einaudi Institute for Economics and Finance
  (EIEF), 2020.

\bibitem[Shi19]{Shim:2019aa}
Eunha Shim, \emph{Optimal dengue vaccination strategies of seropositive
  individuals}, Mathematical Biosciences and Engineering \textbf{16} (2019),
  no.~3, 1171--1189.

\bibitem[SKL{\etalchar{+}}20]{Salje:2020aa}
Henrik Salje, C{\'e}cile~Tran Kiem, No{\'e}mie Lefrancq, No{\'e}mie Courtejoie,
  Paolo Bosetti, Juliette Paireau, Alessio Andronico, Nathana{\"e}l Hoz{\'e},
  Jehanne Richet, Claire-Lise Dubost, et~al., \emph{Estimating the burden of
  {SARS-CoV-2 in France}}, Science (2020).

\bibitem[Wic75]{Wickwire:1975aa}
KH~Wickwire, \emph{Optimal isolation policies for deterministic and stochastic
  epidemics}, Mathematical biosciences \textbf{26} (1975), no.~3-4, 325--346.

\bibitem[Wic79]{Wickwire:1979aa}
K~Wickwire, \emph{Optimal immunization rules for an epidemic with recovery},
  Journal of Optimization Theory and Applications \textbf{27} (1979), no.~4,
  549--570.

\bibitem[YWZZ15]{Yang:2015aa}
Kuan Yang, Ermei Wang, Yinggao Zhou, and Kai Zhou, \emph{Optimal vaccination
  policy and cost analysis for epidemic control in resource-limited settings},
  Kybernetes (2015).

\bibitem[ZWW13]{Zhou:2013aa}
Yinggao Zhou, Jianhong Wu, and Min Wu, \emph{Optimal isolation strategies of
  emerging infectious diseases with limited resources}, Math Biosci Eng MBE
  \textbf{10} (2013), 1691--1701.

\end{thebibliography}

\end{document}